\documentclass[12pt]{amsart}
\usepackage[margin=2.5cm]{geometry}
\usepackage{graphicx}
\usepackage[pagebackref, pdftex]{hyperref}

\renewcommand*{\backref}[1]{}
\renewcommand*{\backrefalt}[4]{
 \ifcase #1
 [No citations.]
 \or [#2]
 \else [#2]
 \fi }

\usepackage{amsmath}
\usepackage{amssymb}
\usepackage{amsthm}
\usepackage{accents}
\usepackage{braket}
\usepackage{cite}
\usepackage{graphicx}
\usepackage{listings}
\usepackage{adjustbox}
\usepackage{framed} 
\usepackage{mdframed}
\usepackage{xcolor}
\usepackage[dvipsnames]{xcolor}
\usepackage{tikz}
\usepackage{pgfplots}
\usetikzlibrary{cd}
\usetikzlibrary{angles}
\usetikzlibrary{shapes.geometric}
\usetikzlibrary{decorations.markings}
\usepackage{color,soul}
\usepackage{bbm}
\usepackage{tabularx}
\makeatletter
\renewcommand{\qed}{$\hfill\square$}

\numberwithin{equation}{section}
\newtheorem{theorem}[equation]{Theorem}

\newtheorem{cor}[equation]{Corollary}

\newtheorem{lem}[equation]{Lemma}

\newtheorem{proposition}[equation]{Proposition}
\newtheorem{prop}[equation]{Proposition}

\newtheorem{definition}[equation]{Definition}

\newtheoremstyle{named}{}{}{\itshape}{}{\bfseries}{.}{.5em}{\thmnote{#3} #1}
\theoremstyle{named}

\newcommand{\refsec}[1]{Section~\ref{Sec:#1}}
\newcommand{\refdef}[1]{Definition~\ref{Def:#1}}
\newcommand{\reffig}[1]{Figure~\ref{Fig:#1}}

\newcommand{\refeqn}[1]{\eqref{Eqn:#1}}
\newcommand{\reflem}[1]{Lemma~\ref{Lem:#1}}
\newcommand{\refprop}[1]{Proposition~\ref{Prop:#1}}
\newcommand{\refthm}[1]{Theorem~\ref{Thm:#1}}

\let\Im\relax
\DeclareMathOperator{\Im}{Im}

\newcommand{\R}{\mathbb{R}}
\newcommand{\C}{\mathbb{C}}
\newcommand{\Z}{\mathbb{Z}}

\newcommand{\To}{\longrightarrow}

\title{An algebraic approach to circle packing}

\author{Daniel V. Mathews}
\address{School of Mathematics, Monash University, VIC 3800, Australia}
\email{{\rm \textcolor{blue}{Daniel.Mathews@monash.edu}}}

\author{Orion Zymaris}
\address{School of Mathematics, Monash University, VIC 3800, Australia}
\email{{\rm \textcolor{blue}{Orion.Zymaris@monash.edu}}}

\begin{document}

\begin{abstract}
We show that for certain triangulations of surfaces, circle packings realising the triangulation can be found by solving a system of polynomial equations. We also present a similar system of equations for unbranched circle packings. The variables in these equations are associated to corners of triangles in the complex, with equations for interior vertices, edges, faces, and generators of first homology. The vertex equations are generalisations of the Descartes circle theorem, of higher degree but more symmetric than those previously found by the authors. We also provide some connections between the spinorial approach of previous work of the authors, and classical Euclidean geometry.
\end{abstract}

\maketitle

\setcounter{tocdepth}{1}
\tableofcontents

\section{Introduction}

\subsection{Overview}

In this paper we present an approach to certain circle packing problems based on solving systems of polynomial equations. For certain $\Delta$-complexes $K$ providing 2-dimensional triangulations of surfaces, we introduce a system of variables and polynomial equations, which we call the \emph{circle packing equations} for $K$. We show that the positive real solutions of this system of equations correspond to certain circle packings realising $K$. The variables correspond to the corners of triangles in $K$. Roughly, the circle packing equations include a \emph{triangle equation} for each triangle of $K$, an \emph{edge equation} for each interior edge of $K$, and a \emph{vertex equation} for each interior vertex of $K$. When $K$ is a torus, there are also two \emph{holonomy equations}.

\begin{theorem}
\label{Thm:circle_packing_eqs_general}
Let $K$ be a $\Delta$-complex triangulating an oriented 2-dimensional surface (possibly with boundary), such that either
\begin{enumerate}
\item
$K$ is a simply connected simplicial complex, or
\item
$K$ triangulates a torus, with universal cover a simplicial complex triangulating $\R^2$.
\end{enumerate}
Then the positive real solutions of a set of circle packing equations for $K$ correspond bijectively to conformal classes of Euclidean or good spherical circle packings realising $K$.
\end{theorem}
We define the requisite notions
as we proceed. 
The hypotheses on $K$ only allow a few topological possibilities. If $K$ is finite, then $K$ triangulates a closed disc $D^2$, sphere $S^2$ or torus $T^2$. 
If $K$ is infinite, it triangulates the plane (equivalently, an open disc), or a closed disc with some points and arcs removed from its boundary, which we call a \emph{semi-open disc}.

Thus, the set of (conformal classes of Euclidean or spherical) circle packings realising $K$ has the structure of a real \emph{semialgebraic set}, given by the circle packing equations and the inequalities that each variable be positive. 
When $K$ is infinite, there are infinitely many variables and infinitely many equations.

When $K$ is an infinite triangulation of a plane surface, the circle packing equations will only find Euclidean circle packings; our equations will not find hyperbolic packings. \refthm{circle_packing_eqs_general} is a clean and convenient statement, but our arguments prove something slightly stronger: some circle packings with orientation reversals correspond to some solutions of the circle packing equations with some variables negative. We explain this slight generalisation as we proceed (see comments at end of \refsec{disc_case}, \refsec{spherical_case_eqns}, \refsec{closed_disc} and \refsec{spheres}).

When $K$ is a plane surface, the circle packing equations are unique, and contain no redundancy. However, for spheres and tori, the circle packing equations depend on choices of additional structures (see \refdef{circle_packing_eqns_sphere} and \refdef{circle_packing_eqns_torus}), so are not unique. Moreover, in these cases, there is some redundancy in the equations, in a sense we can make precise (see \refsec{redundancy_rigidity}). However, any set of circle packing equations has solution set corresponding to circle packings. 

 Finding circle packings via these equations is somewhat analogous to Thurston's approach to finding hyperbolic structures on triangulated 3-manifolds, using gluing and completeness equations \cite{Thurston_notes, Thurston97}. Circle packing theory is known to be closely related to hyperbolic 3-manifolds, at least since Thurston's proof of the Koebe--Andreev--Thurston theorem using hyperbolic geometry in \cite[ch. 13]{Thurston_notes}. The vertex and triangle equations essentially impose Euclidean angle conditions, and edge equations then match lengths. The result is similar to how Thurston gluing equations impose Euclidean conditions on the triangles of a cusp triangulation. Our holonomy equations play a similar role to Thurston's completeness equations.

Our approach to circle packing in this paper can also be regarded as a development of the spinorial approach taken in \cite{MZ25}. Although the results of this paper do not require spinors, they depend crucially on the variables used in that work, there denoted $m_j$, which arose from spinors. (We also give alternative proofs of some of our results here using spinors.) There is a sense, which we will explain, in which our circle packing equations can be regarded allegorically as spinorial ``square roots of geometry". We also provide some connections between this spinorial approach and some classic Euclidean triangle geometry. This connection relies crucially on an observation of Ian Agol \cite{Agol_email}, that the variables $m_j$ introduced in \cite{MZ25} have a simple trigonometric interpretation.

Along the way, we give a new generalisation of the Descartes circle theorem, i.e. an equation relating the curvatures of circles in a flower in the Euclidean plane. This equation has higher degree in the $m_j$ than the equation given in \cite{MZ25} but is more symmetric. The vertex equations in our circle packing equations are precisely these generalised Descartes equations. We also discuss the relationship of these Descartes equations to those of \cite{MZ25}, addressing a minor issue in that paper (not affecting the result) in the process. Recall an \emph{$n$-flower} consists of a central circle $C_\infty$, and $n$ \emph{petal} circles $C_j$, over $j \in \Z/n\Z$, so that the $C_j$ are externally tangent to $C_\infty$ in order around $C_\infty$, and each $C_j$ is externally tangent to $C_{j-1}$ and $C_{j+1}$.

\begin{theorem}[Symmetric generalised Descartes circle theorem]
\label{Thm:gen_Descartes}
Let the circles $C_\infty$ and $C_j$ ($j \in \Z/n\Z$) of an $n$-flower in the Euclidean plane have curvatures $\kappa_\infty, \kappa_j$ respectively. Define $m_j$ for $j \in \Z/n\Z$ as
\begin{equation}
\label{Eqn:mj_def}
m_j = \sqrt{ \left( \frac{\kappa_j}{\kappa_\infty} + 1 \right) \left( \frac{\kappa_{j-1}}{\kappa_\infty} + 1 \right) - 1}.
\end{equation}
Then 
\begin{equation}
\label{Eqn:symmetric_Descartes}
\prod_{j=1}^n \left( m_j + i \right) = \prod_{j=1}^n \left( m_j - i \right).
\end{equation}
\end{theorem}
Similarly to \cite{MZ25}, complex numbers are essential to the proof but a mere convenience in stating equation \refeqn{symmetric_Descartes}. The left and right hand sides of \refeqn{symmetric_Descartes} are complex conjugates of each other, so their equality is equivalent to their imaginary part being zero. The equation above is thus equivalent to
\begin{equation}
\label{Eqn:vertex_eqn}
\sum_{I} (-1)^{\frac{ n-|I|-1 }{2} } \prod_{j \in I} m_{j} = 0,
\end{equation}
where the sum is over subsets $I \subseteq \Z/n\Z$ such that $|I| \equiv n-1$ mod $2$. It therefore has degree $n-1$, but is linear in each variable, and has $\binom{n}{n-1} + \binom{n}{n-3} + \binom{n}{n-5} + \cdots = 2^{n-1}$ terms.

When $n=3$, the equation becomes
\[
m_1 m_2 + m_2 m_3 + m_3 m_1 = 1
\]
from which the classic Descartes equation can be recovered.

\subsection{Circle packings, realisations, and conformal equivalence}
\label{Sec:realisations_equivalence}

Our general reference for circle packing is Stephenson's classic book \cite{StephensonKenneth2005Itcp}.

Let $K$ be a complex as in \refthm{circle_packing_eqs_general}.
 A $\Delta$-complex triangulating an oriented surface can be regarded as a collection of oriented triangles $\widehat{\Delta_j}$ with some pairs of edges glued together by orientation-reversing linear homeomorphisms (see e.g. \cite{Hatcher_text}). We refer to the \emph{triangles of $K$} as the images $\Delta_j$ of the triangles after such identifications. Each triangle $\Delta_j$ has 3 \emph{corners}, which we may define formally as the vertices of $\widehat{\Delta_j}$. A triangle $\Delta_j$ may have some of its edges or vertices identified in $K$, so may have fewer than 3 distinct edges or vertices; however it always has 3 distinct corners.

In case (1) of \refthm{circle_packing_eqs_general}, $K$ is a simplicial complex, so each triangle has distinct vertices; in case (2), this is true of the universal cover $\widetilde{K}$. 
The 1-skeleton of $K$ or $\widetilde{K}$ respectively is a graph without loops or multiple edges. 

 We aim to realise $K$ as the nerve of a circle packing. This circle packing is on an oriented Riemannian surface $\mathcal{S}$ which is: 
\begin{enumerate}
\item[(a)] the Euclidean plane $\R^2$ with its usual orientation, when $K$ is a simply connected plane surface; 
\item[(b)] the standard round unit sphere, with its usual orientation induced by an outward normal, when $K \cong S^2$; or 
\item[(c)] a Euclidean torus, when $K \cong T^2$.
\end{enumerate}
This circle packing will have a circle $C_v$ in $\mathcal{S}$ for each vertex $v$ of $K$. Two circles $C_v, C_w$ are required to be externally tangent when $v,w$ are joined by an edge. For a triangle $\Delta$ of $K$ with vertices $u,v,w$ in anticlockwise order, the mutually externally tangent circles $C_u, C_v, C_w$ are required also to be in anticlockwise order on $\mathcal{S}$. Such a collection of circles yields a map $\Phi \colon K \To \mathcal{S}$, with each vertex $v$ mapping to the centre of $C_v$, edges mapping to geodesic arcs, and each triangle of $K$ mapping to a geodesic triangle in $\mathcal{S}$ by an orientation-preserving homeomorphism. We refer to this map $\Phi$ as a \emph{realisation} of $K$. 
 The image $\Phi(\Delta)$ of $\Delta$ contains arcs of $C_u, C_v, C_w$ bounding sectors of closed discs centred at $u,v,w$ respectively. The complement of these sectors in $\Phi(\Delta)$ is called the \emph{interstice} of $\Delta$ (see Figure \ref{fig:Interstice}).
\begin{figure}[!h]
 \centering
\begin{tikzpicture}[scale=1.4]
 \draw[thick] (0,0)--(2,0)--(1,1.73)--(0,0);
 \filldraw[black] (0,0) -- (1,0) arc (0:60:1) -- cycle;
 \filldraw[black] (2,0) -- (1.5,0.866) arc (120:180:1) -- cycle;
 \filldraw[black] (1,1.73) -- (1.5,1.73/2) arc (-60:-120:1) -- cycle;
\end{tikzpicture} 
\caption{The interstice (in white) of a triangle $\Delta$ in a circle packing.}
 \label{fig:Interstice}
\end{figure}

 Suppose we have two circle packings realising $K$, i.e. two realisation maps $\Phi \colon K \To \mathcal{S}$ and $\Phi' \colon K \To \mathcal{S}'$. If $K$ is a closed disc, sphere, or plane surface, then $\mathcal{S} = \mathcal{S}'$, but if $K \cong T^2$ then $\mathcal{S},\mathcal{S'}$ will in general not be isometric. We say the two circle packings are \emph{conformally equivalent} if there is a conformal bijection $\Psi \colon \mathcal{S} \To \mathcal{S'}$ such that $\Phi' = \Psi \circ \Phi$. This is an equivalence relation on realisations of $K$, and we call the equivalence classes \emph{conformal classes} of circle packings realising $K$. When $K \cong D^2$, such $\Psi$ are just orientation-preserving similarities of the Euclidean plane, so the conformal classes are similarity classes. When $K \cong S^2$, such $\Psi$ are M\"{o}bius transformations. When $K \cong T^2$, such a $\Psi \colon \mathcal{S} \To \mathcal{S}'$ lifts to universal covers as an orientation-preserving conformal automorphism of the Euclidean plane, hence an orientation-preserving similarity.

The triangles around an interior vertex $v$ of $K$ must map under $\Phi$ to triangles around $\Phi(v)$ in the same cyclic order and orientation. Their angles at $\Phi(v)$ must sum to a positive integer multiple $2\pi (\beta + 1)$ of $2\pi$. The non-negative integer $\beta$ is the \emph{branching index} of $\Phi$ at $v$. When $\beta = 0$ at each vertex, we say the realisation $\Phi$ is \emph{unbranched}.

\subsection{Circle packing equations}
\label{Sec:circle_packing_eqns}

 We now define the circle packing equations for a complex $K$. 

We introduce a variable for each corner of each triangle. Denote by $V(K), E(K), F(K)$ the vertices, edges and faces of $K$. Denote by $C(K)$ the corners of triangles of $K$. In general, a corner can be described as a pair $(\Delta, \widehat{v})$, where $\Delta \in F(K)$, and $\widehat{v}$ one of the vertices of its unglued pre-image $\widehat{\Delta}$. When $K$ is a simplicial complex, these corners can be described more simply as pairs $(\Delta, v)$ where $\Delta \in F(K)$ and $v$ is one of its vertices. For us the universal cover $\widetilde{K}$ is always a simplicial complex, so we can also always describe corners as equivalence classes of pairs $(\widetilde{\Delta}, \widetilde{v})$ where $\widetilde{\Delta} \in F(\widetilde{K})$ is a triangle of $\widetilde{K}$ and $\widetilde{v}$ one of its vertices, and two pairs are equivalent if they are related by a deck transformation. We use the notation $(\Delta, \widehat{v})$ when necessary, but prefer the simpler notation $(\Delta, v)$ when possible. We denote the variable corresponding to corner $(\Delta,\widehat{v})$ by $m_{\Delta,\widehat{v}}$, i.e.
\[
\textbf{Variables:} \quad
 m_{\Delta, \widehat{v}} 
\quad \text{for each} \quad 
 (\Delta, \widehat{v}) 
\in C(K).
\]

Now we describe the circle packing equations, which always contain vertex, edge and face equations. Denote by $V_{int}(K)$ and $E_{int}(K)$ the sets of \emph{interior} vertices and edges of $K$, i.e. those not lying entirely on the boundary of $K$. There is a vertex equation for each interior vertex, an edge equation for each interior edge, and a triangle equation for each face of $K$.

Consider an interior vertex $v$ of $K$, i.e. $v \in V_{int}(K)$. As $K$ is topologically a surface, the degree of $v$ is finite; let it be $d$, so there are $d$ corners of triangles around $v$, which we denote $(\Delta_j, v)$. (This notation is somewhat abusive when $K$ is not a simplicial complex, since some of the $\Delta_j$ may coincide; however for us the universal cover $\widetilde{K}$ is always a simplicial complex.) Indeed, in the universal cover we have a 
\emph{combinatorial $d$-flower} consisting of $d$ distinct triangles around $v$; we can denote them in cyclic order as $\Delta_j$ over $j \in\Z/d\Z$. See \reffig{circle_packing_eqn_vars} (left). In any case we have variables $m_{\Delta_j,v}$ over $j \in \Z/d\Z$ corresponding to the $d$ corners at $v$. The vertex equation at $v$ is the generalised Descartes equation of \refeqn{vertex_eqn}, so we have:
\begin{equation}
\label{Eqn:vertex_eqn_list}
\textbf{Vertex equations:} \quad
\sum_{I} (-1)^{\frac{ d-|I|-1 }{2} } \prod_{j \in I} m_{\Delta_j,v} = 0 \quad \text{for each} \quad v \in V_{int}(K),
\end{equation}
where the sum is over subsets $I \subseteq \Z/d\Z$ such that $|I| \equiv d-1$ mod $2$. These equations are perhaps less imposing than the notation suggests: if $d=3$ and the variables are $a,b,c$, then the equation is $ab+bc+ca=1$; if $d=4$ with variables $a,b,c,d$, the equation is $abc+bcd+cda+dab=a+b+c+d$.

\begin{figure}
 \centering
 \begin{tabularx}{\linewidth} { 
 >{\centering\arraybackslash}X 
 >{\centering\arraybackslash}X 
 >{\centering\arraybackslash}X }
 \begin{tikzpicture}[scale=1.5]
 \filldraw[black] (0,0) circle (1pt);
 \draw (0.05,0.2) node {$v$};
 \draw[thick] (0,0)--(1,0);
 \draw[thick] (0,0)--(0.707,0.707);
 \draw[thick] (0,0)--(0.5,-0.866);
 \draw[thick] (0,0)--(-0.5,0.866);
 \draw[thick] (0,0)--(-0.866,-0.5);
 \node[above] at (0.95,0.1) {$\Delta_0$};
 \node[above] at (0.15,0.5) {$\Delta_1$};
 \node[left] at (-0.35,0.1) {$\Delta_2$};
 \node[below] at (-0.25,-0.5) {$\cdots$};
 \node[below] at (0.85,-0.15) {$\Delta_{d-1}$};
 \end{tikzpicture} &
 \begin{tikzpicture}
 \draw[thick] (0,0)--(2,0);
 \draw[thick] (0,0)--(-1.2,1.2);
 \draw[thick] (0,0)--(-1.2,-1.2);
 \draw[thick] (2,0)--(3.2,1.2);
 \draw[thick] (2,0)--(3.2,-1.2);
 \node[above] at (1,0) {$e$};
 \node[left] at (-0.05,0) {$v_1$};
 \node[right] at (2.05,0) {$v_2$};
 \node at (1,1) {$\Delta_1$};
 \node at (1,-1) {$\Delta_2$};
 \end{tikzpicture} &
 \begin{tikzpicture}[scale=1.4]
 \draw[thick] (0,0)--(2,0)--(1,1.6)--(0,0);
 \node[left] at (0,0) {$u$};
 \node[right] at (2,0) {$v$};
 \node[above] at (1,1.6) {$w$};
 \node at (1,0.6) {$\Delta$};
 \end{tikzpicture} \\
 \end{tabularx}
 \caption{Left: an interior vertex $v$. Centre: an interior edge $e$. Right: a face $\Delta$.}
 \label{Fig:circle_packing_eqn_vars}
\end{figure}

Consider an interior edge $e$ of $K$, i.e. $e \in E_{int}(K)$. Let the endpoints of $e$ be $v_1,v_2$, and let the two triangles adjacent to $e$ be $\Delta_1, \Delta_2$. See \reffig{circle_packing_eqn_vars} (centre). We denote the four adjacent corners by $(\Delta_i, v_j)$ over $i=1,2$ and $j=1,2$, accordingly. (Again, this notation is somewhat abusive when $v_1 = v_2$, which may happen when $K$ is not a simplicial complex. However, again since $\widetilde{K}$ is for us always a simplicial complex, $\Delta_1, \Delta_2$ are distinct, and by passing to the universal cover we may assume $v_1, v_2$ are distinct.) In any case we have four variables $m_{\Delta_i,v_j}$ over $i=1,2$ and $j=1,2$, for the four corners adjacent to $e$. Upon them we impose the
\begin{equation}
\label{Eqn:edge_eqn_list}
\textbf{Edge equations:} \quad
m_{\Delta_1, v_1} m_{\Delta_2, v_2} = m_{\Delta_1, v_2} m_{\Delta_2, v_1} \quad \text{for each} \quad e \in E_{int}(K).
\end{equation}

 Consider a triangle $\Delta$ of $K$, i.e. $\Delta \in F(K)$. Let its corners be $u,v,w$, so we have variables $m_{\Delta,u}, m_{\Delta,v}, m_{\Delta,w}$. See \reffig{circle_packing_eqn_vars} (right). Again when $K$ is not a simplicial complex, $u,v,w$ need not be distinct vertices. These variables are required to satisfy the 
\begin{equation}
\label{Eqn:triangle_eqn_list}
\textbf{Triangle equations:} \quad
m_{\Delta,u} m_{\Delta,v} m_{\Delta,w} = m_{\Delta,u} + m_{\Delta,v} + m_{\Delta,w} \quad \text{for each} \quad \Delta \in F(K).
\end{equation}

When $K$ is a plane surface (even an infinite one), this is the complete set of equations. 

\begin{definition}
\label{Def:circle_packing_eqns_disc}
Let $K$ be a simplicial complex triangulating a simply connected oriented plane surface. The \emph{circle packing equations} for $K$ are the vertex equations \refeqn{vertex_eqn_list}, edge equations \refeqn{edge_eqn_list}, and triangle equations \refeqn{triangle_eqn_list}.
\end{definition}

\refthm{circle_packing_eqs_general} asserts that solutions to these equations in positive reals yield circle packings realising $K$.

\subsection{Circle packing equations for spheres}

When $K$ triangulates a 2-sphere, we choose an arbitrary triangle $\Delta_0$ of $K$, which we think of as ``the triangle at the north pole". Removing (the interior of) $\Delta_0$ yields a finite triangulation of a closed disc with three edges on its boundary, which we denote $K_0 = K \setminus \Delta_0$. Provided that the circles avoid the north pole, stereographic projection sends a circle packing of $K$ in $S^2$ to a circle packing of $K_0$ in $\R^2$. We take stereographic projection to project from the north pole $N = (0,0,1)$, through the unit sphere $S^2 \subset \R^3$, to the $xy$-plane in the standard way, hence make the following definition.
\begin{definition}
\label{Def:good_packing}
Suppose $K$ is a simplicial complex triangulating $S^2$, and $\Delta_0$ is a triangle of $K$. A circle packing of $K$ with realisation map $\Phi \colon K \To S^2$ is \emph{good for $\Delta_0$} if $N$ lies in the interstice of $\Delta_0$, and $\Phi^{-1}(N)$ is a single point.
\end{definition}
Thus, in a good packing for $\Delta_0$, every circle and its interior (as defined by $\Phi$) avoids the north pole, and $\Phi|_{K_0}$ has image in $S^2 \setminus \{N\}$. Stereographic projection then yields a circle packing for $K_0$ in the Euclidean plane, where we may regard $\Delta_0$ as the exterior triangle of $K_0$, negatively oriented. 
We define circle packing equations for the pair $(K, \Delta_0)$, rather than just $K$, taking account of the distinguished status of $\Delta_0$.

There is a unique M\"{o}bius transformation sending three mutually tangent circles on $S^2$ to any other three, so to specify a unique conformal class of Euclidean circle packings corresponding to the unique conformal class of circle packings on $S^2$ under stereographic projection, we can specify this exterior triangle $\Delta_0$ to be any fixed similarity class of triangle. For simplicity we choose it to be equilateral, which, because of the negative orientation on $\Delta_0$, corresponds to choosing variables $m_\bullet = -\sqrt{3}$. Thus we have the
\begin{equation}
\label{Eqn:sphere-closing_eqn_list}
\textbf{Sphere-closing equations:} \quad
m_{\Delta_0,v} = - \sqrt{3}
\quad \text{for each corner $v$ of $\Delta_0$.}
\end{equation}
The sphere-closing equations imply the triangle equation for $\Delta_0$, which then becomes redundant. (We could equivalently take the sphere-closing equations to set the three variables in the corners of $\Delta_0$ to any three negative constants satisfying the triangle equation.) These variables being set to negative constants, a ``positive" solution to the circle packing equations means all other variables are positive; we effectively regard the variables for corners of $\Delta_0$ as constants. 

\begin{definition}
\label{Def:circle_packing_eqns_sphere}
Let $K$ be a complex triangulating an oriented 2-sphere, and $\Delta_0$ a triangle of $K$. The \emph{circle packing equations} for $(K, \Delta_0)$ consist of the vertex equations \refeqn{vertex_eqn_list}, the edge equations \refeqn{edge_eqn_list}, the triangle equations \refeqn{triangle_eqn_list}, and the sphere-closing equations \refeqn{sphere-closing_eqn_list}.
A set of circle packing equations for $K$ is a set of circle packing equations for $(K, \Delta_0)$, for some $\Delta_0$.
\end{definition}

A precise version of \refthm{circle_packing_eqs_general} in the spherical case is then as follows.
\begin{theorem}
\label{Thm:circle_packing_eqs_spherical}
Let $K$ be a simplicial complex triangulating an oriented $2$-sphere, and $\Delta_0$ a triangle of $K$. 
Real solutions of the circle packing equations for $(K, \Delta_0)$, such that all $m_{\Delta,v} > 0$ whenever $\Delta \neq \Delta_0$, correspond bijectively to conformal classes of spherical circle packings realising $K$ which are good for $\Delta_0$.
\end{theorem}

We will see that, after imposing all the circle packing equations for $K_0$, any two of the three edge equations for the edges of $\Delta_0$ imply the third, as well as the vertex equations at the vertices of $\Delta_0$. We define a \emph{reduced} set of circle packing equations in \refdef{reduced_eqns_S2}, which will be irredundant and equivalent to the full set.

\subsection{Circle packing equations for tori}
\label{Sec:circle_packing_eqns_tori}

In the toroidal case, we need to consider certain closed curves. In a triangle $\Delta$ of a complex, a \emph{normal arc} is an arc properly embedded in $\Delta$, each of whose endpoints is an interior point of a side of $\Delta$, with its two endpoints on distinct sides of $\Delta$. Define a \emph{normal curve} in a complex $K$ to be a curve $\gamma$ in $K$ which is obtained by joining a sequence of normal arcs, end to end. We consider both \emph{closed normal curves}, which join a finite sequence of normal arcs in cyclic order, and \emph{open normal curves}, which join a sequence (finite or countably infinite) of normal arcs in a total order. (Such curves are ``normal" in the sense of normal surface theory.) Any simple closed curve $\gamma$ on $K$ is homotopic to a closed normal curve: a homotopy makes $\gamma$ intersect the 1-skeleton of $K$ transversely at interior points of edges, and a further homotopy of ``finger moves" removes ``backtracking" intersections of $\gamma$ with edges of $K$, yielding a normal curve (or a contractible curve in a single triangle, which can straightforwardly be made normal).

Suppose $K \cong T^2$, and consider two oriented closed curves $\lambda, \mu$ generating $H_1 (T^2)$, which after a homotopy we may assume are closed normal curves. We regard $\lambda$ as a sequence of oriented normal arcs $\lambda_j$, over $j \in \Z/n_\lambda \Z$ for some positive integer $n_\lambda$, concatenated in cyclic order. Each normal arc $\lambda_j$ lies in a triangle $\Delta^\lambda_j$ of $K$, with endpoints on two sides of $\Delta^\lambda_j$. These two sides meet at a unique corner of $\Delta$, which we denote $c^\lambda_j$. (Note as $K$ need not be a simplicial complex, all corners of $\Delta$ may be at the same vertex of $K$.) We say $\lambda_j$ \emph{cuts off} the corner $c^\lambda_j$. The triangle $\Delta^\lambda_j$ inherits an orientation from $K$, and $\lambda_j$ inherits an orientation from $\lambda$, so $\lambda_j$ proceeds around $c^\lambda_j$ in an anticlockwise or clockwise direction. We define a number $\delta^\lambda_j$ as $+1$ or $-1$ accordingly as $\lambda_j$ cuts off $c^\lambda_j$ in an anticlockwise or clockwise direction. See \reffig{normal_arcs}.

\begin{figure}
 \centering
 \begin{tikzpicture}[scale=1.8]
 \draw[thick] (0,0)--(2,0)--(1,1.6)--(0,0);
				\draw[thick] (1,1.6)--(2.5,1.6)--(2,0);
				\draw[thick] (2.5,1.6)--(4,1.6)--(2,0);
				\draw (1,0.2) node {$\Delta^\lambda_1$};
				\draw (1.75,1.4) node {$\Delta^\lambda_2$};
				\draw (3.1,1.4) node {$\Delta^\lambda_3$};
				\draw[ultra thick, ->] (-0.3,0.8)--(4.1,0.8);
				\draw (4.3,0.8) node {$\lambda$};
				\draw (1,0.6) node {$\lambda_1$};
				\draw (1.9,0.95) node {$\lambda_2$};
				\draw (2.7,0.95) node {$\lambda_3$};
 \filldraw[black] (1,1.6) circle (1pt);
				\draw (1,1.3) node {$c^\lambda_1$};
				\filldraw[black] (2,0) circle (1pt);
				\draw (1.95,0.4) node {$c^\lambda_2$};
				\draw (2.3,0.45) node {$c^\lambda_3$};
 \end{tikzpicture} \\
 \caption{Normal arcs along a normal curve $\lambda$. The normal arcs $\lambda_1, \lambda_2, \lambda_3$ cut off corners $c^\lambda_1, c^\lambda_2, c^\lambda_3$ of triangles $\Delta^\lambda_1, \Delta^\lambda_2, \Delta^\lambda_3$ respectively. Since $\lambda_1$ proceeds anticlockwise around $c^\lambda_1$, and $\lambda_2$ and $\lambda_3$ proceed clockwise around $c^\lambda_2$ and $c^\lambda_3$, we have $\delta^\lambda_1 = 1$ and $\delta^\lambda_2 = \delta^\lambda_3 = -1$.}
 \label{Fig:normal_arcs}
\end{figure}

Similarly, choosing a basepoint on $\mu$, we can regard $\mu$ as a cyclic sequence of normal arcs $\mu_j$ over $j \in \Z/n_\mu \Z$ for some positive integer $n_\mu$. Each $\mu_j$ then cuts off a corner $c^\mu_j$ in a triangle $\Delta^\mu_j$ of $K$, and we define $\delta^\mu_j$ as $+1$ or $-1$ accordingly as $\mu_j$ cuts off $c^\mu_j$ in anticlockwise or clockwise fashion. As we will see, the holonomy equations then require that both
\[
\prod_{j=1}^{n_\lambda} \left( m_{\Delta^\lambda_j,c^\lambda_j} + i \right)^{\delta^\lambda_j}
\quad \text{and} \quad
\prod_{j=1}^{n_\mu} \left( m_{\Delta^\mu_j,c^\mu_j} + i \right)^{\delta^\mu_j} 
\]
be real. Since, for real $m$, we have $(m+i)^{-1}$ is a positive multiple of $m-i$, the factors $(m_\bullet +i)^{\delta}$ can be replaced by $(m_\bullet + i \delta )$. Expanding out the imaginary part as zero, the final circle packing equations in the torus case are the
\begin{align}
\label{Eqn:holonomy_eqns_1}
\textbf{Holonomy equations:}& \quad
\sum_{I_\lambda} 
\left( -1 \right)^{\frac{n_\lambda - |I_\lambda| - 1}{2}}
\prod_{j \in I_\lambda} m_{\Delta^\lambda_j, c^\lambda_j} 
\prod_{j \in I'_\lambda} \delta^\lambda_j = 0 \\
\label{Eqn:holonomy_eqns_2}
\quad \text{and}& \quad
\sum_{I_\mu} 
\left( -1 \right)^{\frac{n_\mu - |I_\mu| - 1}{2}}
\prod_{j \in I_\mu} m_{\Delta^\mu_j, c^\mu_j} 
\prod_{j \in I'_\mu} \delta^\mu_j = 0.
\end{align}
Here the first sum is over subsets $I_\lambda \subseteq \Z/n_\lambda \Z$ such that $|I_\lambda| \equiv n_\lambda - 1$ mod $2$, and $I'_\lambda$ is the complement of $I_\lambda$, i.e. $I'_\lambda = (\Z/n_\lambda \Z) \setminus I_\lambda$. Similarly, the second sum is over $I_\mu \subseteq \Z/n_\mu \Z$ such that $|I_\mu| \equiv n_\mu - 1$ mod $2$, and $I'_\mu = (\Z/n_\mu \Z) \setminus I_\mu$.
As with the vertex equations, the holonomy equations are perhaps less imposing than they appear; we will see an example in \refsec{standard_torus}.

\begin{definition}
\label{Def:circle_packing_eqns_torus}
Let $K$ be a complex triangulating an oriented torus.
Let $\lambda, \mu$ be oriented closed normal curves forming a basis for $H_1 (K)$. The \emph{circle packing equations} for $(K, \lambda, \mu)$ are the vertex equations \refeqn{vertex_eqn_list}, the edge equations \refeqn{edge_eqn_list}, the triangle equations \refeqn{triangle_eqn_list}, and the holonomy equations \refeqn{holonomy_eqns_1} and \refeqn{holonomy_eqns_2}. A set of circle packing equations for $K$ is a set of circle packing equations for $(K, \lambda, \mu)$, for some $\lambda, \mu$.
\end{definition}

 \refthm{circle_packing_eqs_general} asserts that solutions to a set of circle packing equations for $K$ in positive reals yield circle packings realising $K$.

When $K$ is planar, the set of circle packing equations is unique. However when $K$ is a sphere or torus, the choices of $\Delta_0$ or $\lambda, \mu$ mean that the set of circle packing equations is not unique. However, any such choice provides us with a set of circle packing equations to which the theorem applies. 

\subsection{ Unbranched circle packing equations}
If we are interested in \emph{unbranched} realisations of $K$, then we may replace the vertex equations with \emph{unbranched vertex equations} as follows. The vertex equations, in the form of the generalised Descartes equation \refeqn{symmetric_Descartes}, assert that $\prod_{j=1}^d (m_j + i)$ is real, so that its argument is an integer multiple of $\pi$. The unbranched condition is that this argument is exactly $\pi$,
\begin{equation}
\label{Eqn:unbranched_vertex_eqn_list} 
\textbf{Unbranched vertex equations:} \quad
\sum_{j=1}^d \arg \left( m_{\Delta_j,v} + i \right) = \pi, \quad
\text{for each} \quad v \in V_{int}(K).
\end{equation}

Similarly for the holonomy equations, they assert that $\prod \left( m_\bullet + i \delta_\bullet \right)$ is real, so that the argument is an integer multiple of $\pi$, and the unbranched condition is that this argument is exactly $0$,
\begin{align}
\label{Eqn:unbranched_holonomy_eqns_1}
\textbf{Unbranched holonomy equations:}& \quad
\sum_{j=1}^{n_\lambda} \arg \left( m_{\Delta^\lambda_j,c^\lambda_j} + i \delta^\lambda_j \right) = 0 \\
\label{Eqn:unbranched_holonomy_eqns_2}
\quad \text{and}& \quad
\sum_{j=1}^{n_\mu} \arg \left( m_{\Delta^\mu_j, c^\mu_j} + i \delta^\mu_j \right) = 0.
\end{align}
Here and throughout, we take the principal value of the argument of a nonzero complex number to lie in $(-\pi, \pi]$. A set of \emph{unbranched circle packing equations} for $K$ is then be defined as a set of circle packing equations for $K$, as in \refdef{circle_packing_eqns_disc}, \refdef{circle_packing_eqns_sphere} or \refdef{circle_packing_eqns_torus}, but with vertex equations \refeqn{vertex_eqn_list} replaced by unbranched vertex equations \refeqn{unbranched_vertex_eqn_list} and, if necessary, holonomy equations \refeqn{holonomy_eqns_1} and \refeqn{holonomy_eqns_2} replaced by the unbranched holonomy equations \refeqn{unbranched_holonomy_eqns_1} and \refeqn{unbranched_holonomy_eqns_2}. Of course, unlike the other equations, the unbranched vertex and holonomy equations are not polynomial equations in the variables $m_\bullet$.

 \refthm{circle_packing_eqs_general} then has an unbranched version as follows.

\begin{theorem}
\label{Thm:circle_packing_eqs_general_unbranched}
Let $K$ be a complex satisfying the hypotheses of \refthm{circle_packing_eqs_general}. Then the positive real solutions of a set of unbranched circle packing equations for $K$ correspond bijectively to conformal classes of unbranched Euclidean or unbranched good spherical circle packings realising $K$.
\end{theorem}

\subsection{Geometric meaning of the circle packing equations}

The variables $m_{\Delta,v}$ in our circle packing equations are related to curvatures of circles by equations like \refeqn{mj_def}, but they also have a straightforward geometric interpretation: they are cotangents of half the angles in the corresponding corners of Euclidean triangles in the map $\Phi$ realising $K$. (In the case $K \cong S^2$, these are triangles in the Euclidean plane after stereographic projection.)

In \cite{MZ25}, the variables $m_j$ arose from spinorial considerations, as follows. Let $K$ be a combinatorial $n$-flower, with central vertex $v_\infty$ and other vertices $v_j$ over $j \in \Z/n\Z$ in anticlockwise cyclic order. A circle packing realising $K$, after inversion in the central circle, can be regarded as a collection of horocycles in the conformal disc model of the hyperbolic plane, and by the first author's work \cite{M_spinors_horospheres} these can be associated to spinors $(\xi, \eta) \in \R^2$. Thus we obtain spinors $(\xi_j, \eta_j)$ for each $j \in \Z/n\Z$, and we define associated complex variables $z_j = \xi_j + i \eta_j$. Then the $m_j$ arose as the real part of $z_{j-1} \overline{z_j}$.

Agol's observation \cite{Agol_email} leads to the equality
\[
m_j = \cot \left( \frac{\theta_j}{2} \right),
\]
where $\theta_j$ is the Euclidean angle subtended at $\Phi(v_\infty)$ by $\Phi(v_{j-1})$ and $\Phi(v_j)$.

More generally, when we have $K$ as in \refthm{circle_packing_eqs_general} triangulating an open or closed disc or torus, then in a realisation $\Phi \colon K \To \mathcal{S}$, the Riemannian surface $\mathcal{S}$ is Euclidean, and each triangle $\Delta$ of $K$ is realised as a Euclidean triangle. We denote by $\theta_{\Delta,\widehat{v}}$ the angle in the corner $(\Delta, \widehat{v})$. Then we have the following.
\begin{prop}
\label{Prop:m_meaning}
 
Suppose $K$ is a $\Delta$-complex satisfying the hypotheses of \refthm{circle_packing_eqs_general}, and $\{m_{\Delta,\widehat{v}}\}_{(\Delta,\widehat{v}) \in C(K)}$ is a positive real solution of a set of circle packing equations for $K$, corresponding by \refthm{circle_packing_eqs_general} to a circle packing with realisation $\Phi \colon K \To \mathcal{S}$, where $\mathcal{S}$ is Euclidean. Let this circle packing have curvatures $\{\kappa_v\}_{v \in V(K)}$ and angles $\{\theta_{\Delta,\widehat{v}}\}_{(\Delta,\widehat{v}) \in C(K)}$. Then each $m_{\Delta,\widehat{v}}$ is related to the angles $\theta_{\Delta,\widehat{v}}$ by
\[
 m_{\Delta,\widehat{v}}
 = \cot \left( \frac{ \theta_{\Delta,\widehat{v}} } {2} \right),
\]
and to the curvatures $\kappa_v$ by
\begin{equation}
\label{Eqn:ms_in_terms_of_ks}
 m_{\Delta,\widehat{v}} = \sqrt{ \left( \frac{\kappa_u}{\kappa_v} + 1 \right) \left( \frac{\kappa_w}{\kappa_v} + 1 \right) - 1}
= \frac{\sqrt{\kappa_u \kappa_v + \kappa_v \kappa_w + \kappa_w \kappa_u}}{\kappa_v}.
\end{equation}
where $\widehat{u},\widehat{v},\widehat{w}$ are the vertices of $\widehat{\Delta}$, and $u,v,w$ the corresponding vertices of $\Delta$.
\end{prop}
 The same relations hold when $K \cong S^2$, after stereographic projection to $\R^2$, as we discuss in \refsec{spherical_case_eqns} and \refsec{spheres}.
The same relations also hold for solutions of the unbranched circle packing equations, since they form a subset of solutions of the circle packing equations.

As we will see, the vertex equations can be interpreted as saying that the angles around a vertex sum to a multiple of $2\pi$; the unbranched vertex equations as saying that angles sum to exactly $2\pi$; and the triangle equations as saying that the angles in a triangle sum to $\pi$. We prove this in \refsec{angles_cotangents}. In other words, angle sum conditions on the angles $\theta_\bullet$ yield polynomial conditions in the variables $m_\bullet = \cot(\theta_\bullet/2)$.

The edge equations can also be interpreted via elementary Euclidean geometry, relating the cotangents of half angles in a triangle to the curvatures of its Soddy circles. The relevant Euclidean fact is \reflem{edge_eqn_fact} and we apply it to the edge equations in \reflem{edge_eqns}.

There are several senses in which the circle packing equations are, roughly speaking, ``spinorial". In addition to the variables $m_\bullet$ arising naturally from spinors as in \cite{MZ25}, it is clear from the above that the \emph{half}-angles in the realisation $\Phi$ are crucial variables. These half angles arise quite literally from square roots of complex numbers. As we will see in \reflem{exp_Agol}, for an angle $\theta$ and its corresponding $m$ variable,
\[
e^{i\theta} = \frac{m + i}{m - i}.
\]
The numerator and denominator being conjugates, it follows that $\arg(m + i) = -\arg(m - i) = \theta/2$, so that each $m + i$ is, up to a positive real multiple, a square root of $e^{i\theta}$. The vertex and triangle equations are derived straightforwardly from elementary complex number geometry using these $(m + i)$.

From the variables $m_\bullet$ in the circle packing equations, the curvatures $\kappa_v$ of the circles $C_v$ can be recovered, but only up to an overall constant. Indeed, as seen in \refeqn{ms_in_terms_of_ks}, each $m_\bullet$ only depends on the \emph{ratios} of curvatures. If all $m_\bullet$ are known, then along an edge $e$ with vertices $v,w$, letting $\Delta$ be a triangle adjacent to $e$, we have corners $(\Delta,v), (\Delta,w)$ and the equality
\[
\frac{m_{\Delta,v}}{m_{\Delta,w}} = \frac{\kappa_w}{\kappa_v}.
\]
This follows straightforwardly from using the final expression of \refeqn{ms_in_terms_of_ks} for each $m_\bullet$. Indeed, the edge equations \refeqn{edge_eqn_list} can be interpreted as saying that this ratio of curvatures should not depend on the choice of $\Delta$ adjacent to $e$. The fact that the $m_\bullet$ only determine the $\kappa_v$ up to an overall factor is why \refthm{circle_packing_eqs_general} refers to \emph{conformal} classes of circle packings. We can say that the curvatures $\{\kappa_v\}_{v \in V}$ are determined up to an overall factor, or as a point of a projective space $\R_+ P^{n-1}$, which is $\R_+^n$ modulo scaling by positive factors, and homeomorphic to $\R_+^{n-1}$.

 When $K$ is a closed disc, by the boundary value theorem of circle packing theory (e.g. \cite[thm. 11.6]{StephensonKenneth2005Itcp}), after specifying the curvatures of boundary circles and branch structure (see e.g. \cite[def. 11.4]{StephensonKenneth2005Itcp}) at each internal vertex, there is a circle packing realising $K$, unique up to isometry. Combining this with \refthm{circle_packing_eqs_general} we obtain the following.
\begin{cor}
 Suppose $K$ is a closed disc. Let the boundary vertices of $K$ be $v_1, \ldots, v_n$, let $\beta$ be a branch structure on $K$, and let $(c_1, \ldots, c_n) \in \R_+^n$. Then there exists a unique positive real solution $\{m_{\Delta,\widehat{v}}\}_{(\Delta,\widehat{v}) \in C(K)}$ of the circle packing equations for $K$ such that the curvatures $\{\kappa_v\}_{v \in V(K)}$ determined by the $m_{\Delta,\widehat{v}}$ satisfy $[\kappa_{v_1} : \cdots : \kappa_{v_n}] = [c_1 : \cdots : c_n]$.
\qed
\end{cor}

Thus, for each fixed branch structure on $K$, the space of real positive solutions of the circle packing equations for $K$ is isomorphic to $\R_+ P^{n-1} \cong \R_+^{n-1}$, where $n$ is the number of boundary vertices of $K$. The overall space of real positive solutions is isomorphic to a disjoint union of copies of $\R_+^{n-1}$, one for each branch structure.

\subsection{ Redundancy and rigidity in circle packing equations}
\label{Sec:redundancy_rigidity}
 The above considerations imply
that the circle packing equations for a closed disc contain no redundancy; all the equations are independent. Indeed, letting $V=|V(K)|$, $E=|E(K)|$ and $F = |F(K)|$, the number of variables is $3F$, the number of vertex equations is $|V_{int}(K)| = V - n$, the number of edge equations is $|E_{int}(K)| = E - n$, and the number of triangle equations is $F$. So the number of variables minus equations is
\[
3F - \left( V + E + F - 2n \right)
= \left( -V + E - F \right) - 2E + 3F + 2n 
= n-1,
\]
the dimension of the space of solutions. The last equality uses Euler's formula and the fact that, since each face of $K$ is a triangle, $3F = 2E - n$.

When $K \cong S^2$, the ``fundamental theorem" or ``discrete uniformisation theorem" of circle packing theory (e.g. \cite[thm. 4.3]{StephensonKenneth2005Itcp}) tells us that $K$ has a unique univalent circle packing, up to conformal equivalence. This will provide us with a distinguished solution to the circle packing equations for $K$ (in their standard or unbranched versions) in these cases.

As it turns out, when $K \cong S^2$ a set of circle packing equations has redundancy. In particular, one edge equation from an edge $e_0$ of $\Delta_0$, the triangle equation for $\Delta_0$, and the vertex equations of $\Delta_0$, can be removed without changing the solution set; hence the following.
\begin{definition}
\label{Def:reduced_eqns_S2}
Suppose $K$ is a simplicial complex triangulating $S^2$, $\Delta_0$ is a triangle of $K$, and $e_0$ is an edge of $\Delta_0$. The \emph{reduced circle packing equations} for $(K, \Delta_0, e_0)$ consist of the vertex equations \refeqn{vertex_eqn_list} over all vertices disjoint from $\Delta_0$, the edge equations \refeqn{edge_eqn_list} over all edges other than $e_0$, the triangle equations \refeqn{triangle_eqn_list} over all triangles other than $\Delta_0$, and the sphere-closing equations \refeqn{sphere-closing_eqn_list}.
\end{definition}

\begin{proposition}
\label{Prop:circle_packing_reduction}
The positive solutions of the circle packing equations for $(K, \Delta_0)$ are identical to those of the reduced circle packing equations for $(K, \Delta_0, e_0)$.
\end{proposition}

We also observe rigidity in our equations in the spherical case. Thus, regarding labels in $\Delta_0$ as constants, the reduced circle packing equations for $(K, \Delta_0, e_0)$ have $3F-3$ variables, $V-3$ vertex equations, $E-1$ edge equations, and $F-1$ triangle equations. So the number of variables minus equations is
\[
\left( 3F-3 \right) - \left( V+E+F - 5 \right)
= -V+E-F + (3F-2E) +2
= 0,
\]
where the last equality uses Euler's formula $V-E+F=2$ and $3F=2E$.

When $K \cong T^2$, the fundamental theorem again applies and we again observe rigidity in our equations. There is again redundancy in the equations, and we can remove an edge and vertex equation on the boundary of a fundamental domain as follows. (We discuss fundamental domains in more detail in \refsec{torus_eqns_satisfied}.)
\begin{definition}
\label{Def:reduced_eqns_torus}
Suppose $K$ triangulates $T^2$, with universal cover a simplicial complex triangulating $\R^2$. Let $\lambda, \mu$ be oriented closed normal curves forming a basis for $H_1 (K)$. Let $v_0 \in V(K)$. Let $e_0 \in E(K)$ be an edge on the boundary of a fundamental domain complex for $K$. The \emph{reduced circle packing equations} for $(K, \lambda, \mu, e_0, v_0)$ consist of the vertex equations \refeqn{vertex_eqn_list} over all vertices other than $v_0$, the edge equations \refeqn{edge_eqn_list} over all edges other than $e_0$, the triangle equations \refeqn{triangle_eqn_list}, and the holonomy equations \refeqn{holonomy_eqns_1} and \refeqn{holonomy_eqns_2}.
\end{definition}

\begin{proposition}
\label{Prop:circle_packing_reduction_torus}
The positive solutions of the circle packing equations for $(K, \lambda, \mu)$ are identical to those of the reduced circle packing equations for $(K, \lambda, \mu, e_0, v_0)$.
\end{proposition}

The reduced circle packing equations for $(K, \lambda, \mu, e_0, v_0)$ have $3F$ variables, $V-1$ vertex equations, $E-1$ edge equations, $F$ triangle equations, and $2$ holonomy equations, so the number of variables minus equations is
\[
3F - (V+E+F) 
= -V+E-F + (3F-2E) = 0,
\]
again using $3F=2E$ and $V-E+F=0$.

The circle packing equations are analogous to the Thurston gluing equations in several ways. Given an ideal triangulation of a hyperbolic 3-manifold, a cusp triangulation is a finite simplicial complex triangulating a torus. Thurston associates a variable to each corner of each triangle (e.g. $z_1, z_2, z_3$ and $w_1, w_2, w_3$ in the figure-8 knot complement example of \cite[ch. 4]{Thurston_notes}). The three variables in a triangle satisfy relations, and the variables around a vertex satisfy gluing equations multiplying to $1$, with an alternative ``unbranched" logarithmic formulation. By Mostow rigidity, there is a unique solution. Solutions to the circle packing equations are similarly unique and rigid once boundary and branching conditions are specified. Similar relationships are also discussed in Thurston's notes \cite[ch. 13]{Thurston_notes}.

\subsection{Examples}

\subsubsection{Descartes configuration}

Consider the standard Descartes 3-flower configuration. Let $K$ be the 3-flower complex, with corners labelled as in \reffig{CirclePackEqsExample}. (Our variables skip $i$ to avoid confusion with $\sqrt{-1}$.)
The circle packing equations are then

\begin{figure}[!h]
 \centering
 \begin{tikzpicture}[scale=2]
 \coordinate (A) at (0,1.1);
 \coordinate (B) at (-1.3,-0.9); 
 \coordinate (C) at (1.6,-0.9);
 \coordinate (O) at (0,0);
 \draw[thick] (O) -- (A) node[above] {$1$};
 \draw[thick] (O) -- (B) node[left] {$2$};
 \draw[thick] (O) -- (C) node[right] {$3$};
 \draw[thick] (A) -- (B);
 \draw[thick] (B) -- (C);
 \draw[thick] (C) -- (A);
 \pic [draw, ->, "$a$", draw=none, angle eccentricity=0.7] {angle = A--O--B};
 \pic [draw, ->, "$b$", draw=none, angle eccentricity=1.2] {angle = B--A--O};
 \pic [draw, ->, "$c$", draw=none, angle eccentricity=1.5] {angle = O--B--A};
 \pic [draw, ->, "$d$", draw=none, angle eccentricity=0.7] {angle = B--O--C};
 \pic [draw, ->, "$e$", draw=none, angle eccentricity=1.2] {angle = C--B--O};
 \pic [draw, ->, "$f$", draw=none, angle eccentricity=2] {angle = O--C--B};
 \pic [draw, ->, "$g$", draw=none, angle eccentricity=0.64] {angle = C--O--A};
 \pic [draw, ->, "$h$", draw=none, angle eccentricity=1.95] {angle = A--C--O};
 \pic [draw, ->, "$j$", draw=none, angle eccentricity=1.2] {angle = O--A--C};
 \end{tikzpicture}
 \caption{Complex of a 3-flower.}
 \label{Fig:CirclePackEqsExample}
\end{figure}

\[
\begin{array}{c|c|c}
\textbf{Triangle equations} & \textbf{Edge equations} & \textbf{Vertex equations} \\
\hline
\begin{array}{rcl}
abc &=& a+b+c \\
def &=& d+e+f \\
ghj &=& g+h+j
\end{array}
&
\begin{array}{rcl}
ae &=& cd \\
aj &=& bg \\
dh &=& fg 
\end{array}
&
\begin{array}{rcl}
1 &=& ad+dg+ga,
\end{array}
\end{array}
\]
This system of $7$ equations in $9$ variables has a 2-parameter set of solutions. Indeed, we can set $a,d$ freely, and then the other corner variables can be solved as rational functions of $a$ and $d$. (Other pairs of variables can be set freely too: our choice of $a$ and $d$ is quite arbitrary.) 

We obtain
\begin{gather*}
b = \frac{a(a+d)}{1-a-d+a^2}, \quad
c = \frac{a}{a+d-1}, \quad
e = \frac{d}{a+d-1}, \quad
f = \frac{d(a+d)}{1-a-d+d^2} \\
g = \frac{1-ad}{a+d}, \quad
h = \frac{1-ad}{1-a-d+d^2}, \quad
j = \frac{1-ad}{1-a-d+a^2}.
\end{gather*}
(In solving, we made a choice of positive solution: successively solving for $g,e,f,h,j,b$ and then $c$, we obtain $c = a/(-1 \pm (a+d))$, of which one choice is manifestly negative and discarded.)

This gives an explicit parametrisation of the space of circle packings of $K$. The relative curvatures can be recovered, for example as
\[
\frac{\kappa_1}{\kappa_\infty} = \frac{a}{b} = \frac{1-a-d+a^2}{a+d}, \quad
\frac{\kappa_2}{\kappa_\infty} = \frac{a}{c} = a+d-1, \quad
\frac{\kappa_3}{\kappa_\infty} = \frac{d}{f} = \frac{1-a-d+d^2}{a+d},
\]
from which one can verify the classic Descartes equation.

\subsubsection{Tetrahedron}

Suppose $K$ is a tetrahedron. Such a $K$ is a finite simplicial complex, with $4$ triangles, triangulating a sphere. Removing any triangle $\Delta_0$ yields the Descartes 3-flower complex above, which we now denote $K_0$, adopting the same labels. We add labels in the corners of $\Delta_0$ and set them to $-\sqrt{3}$ by the sphere-closing equations. Letting $e_0$ be the edge from $2$ to $3$ in \reffig{CirclePackEqsExample}, the reduced circle packing equations for $(K, \Delta_0, e_0)$ are precisely those of the previous example, together with the two further edge equations
\[
-b \sqrt{3} = - c \sqrt{3}
\quad \text{and} \quad
-h \sqrt{3} = - j \sqrt{3},
\]
i.e. $b=c$ and $h=j$. Following the parametrisation of solutions for the Descartes configuration above, and recalling we only consider positive solutions, this yields $a=d = 1/\sqrt{3}$, leading to the unique solution 
\[
a=d=g=\frac{1}{\sqrt{3}} = \cot \frac{\pi}{3}, \quad
b=c=e=f=h=j = 2+\sqrt{3} = \cot \frac{\pi}{12},
\]
corresponding to \reffig{CirclePackEqsExample} being an equilateral triangle split into three congruent triangles at its centroid.

Applying stereographic projection yields $4$ mutually tangent circles on the sphere in a regular tetrahedral arrangement. Any circle packing realising $K$ on $S^2$ is related to this one by a conformal automorphism of $S^2$.

\subsubsection{Standard torus}
\label{Sec:standard_torus}
Let $K$ be the standard triangulation of a torus as in \reffig{CirclePackEqsTorusExample}, with edge identifications indicated, as well as oriented simple closed curves $\lambda, \mu$ forming a basis for $H_1 (K)$. We denote vertices of unglued triangles, and hence corners of triangles, by $1$--$4$ as shown. Although $K$ is not a simplicial complex, its universal cover is a simplicial complex triangulating $\R^2$, hence it satisfies the hypotheses of \refthm{circle_packing_eqs_general}.

\begin{figure}[!h]
 \centering
 \begin{tikzpicture}[scale=2]
 \coordinate (A) at (0,0);
 \coordinate (B) at (2,0); 
 \coordinate (C) at (2,2);
 \coordinate (D) at (0,2);
				\begin{scope}[decoration={
						markings,
						mark=at position 0.5 with {\arrow{>}}}
						] 
						\draw[thick, postaction={decorate}] (A) -- (B) node[below right] {$2$};
						\draw[thick, postaction={decorate}] (D) node[above left] {$4$} -- (C);
				\end{scope}				
				\begin{scope}[decoration={
						markings,
						mark=at position 0.5 with {\arrow{>>}}}
						] 
						\draw[thick, postaction={decorate}] (B) -- (C) node[above right] {$3$};
						\draw[thick, postaction={decorate}] (A) node[below left] {$1$} -- (D);
				\end{scope}
				\draw[thick] (A) -- (C);
 \pic [draw, ->, "$a$", draw=none, angle eccentricity=1.2] {angle = B--A--C};
 \pic [draw, ->, "$b$", draw=none, angle eccentricity=1] {angle = C--B--A};
 \pic [draw, ->, "$c$", draw=none, angle eccentricity=1.5] {angle = A--C--B};
 \pic [draw, ->, "$d$", draw=none, angle eccentricity=1.5] {angle = C--A--D};
 \pic [draw, ->, "$e$", draw=none, angle eccentricity=1.5] {angle = D--C--A};
 \pic [draw, ->, "$f$", draw=none, angle eccentricity=1] {angle = A--D--C};
				\draw[ultra thick, ->] (-0.2,1.4) -- (2.2,1.4) node [right] {$\lambda$};
				\draw[ultra thick, ->] (0.6,-0.2) -- (0.6,2.2) node [above] {$\mu$};
 \end{tikzpicture}
 \caption{Complex of a torus.}
 \label{Fig:CirclePackEqsTorusExample}
\end{figure}

The circle packing equations are then
\[
\begin{array}{c|c|c}
\textbf{Triangle equations} & \textbf{Edge equations} & \textbf{Holonomy equations} \\
\hline
\begin{array}{rcl}
abc &=& a+b+c \\
def &=& d+e+f \\
\end{array}
&
\begin{array}{rcl}
ae &=& cd \\
ae &=& bf \\
cd &=& bf 
\end{array}
&
\begin{array}{rcl}
c-d &=& 0 \\
a-e &=& 0
\end{array}
\end{array}
\]
as well as the vertex equation
\[
\left( abcde+\cdots 
+bcdef \right)
- \left( abc+\cdots+def \right)
+ \left( a+b+c+d+e+f \right) = 0.
\]
Here the first bracket contains all $\binom{6}{5}=6$ products of $5$ distinct variables from $\{a,b,c,d,e,f\}$, and the second bracket contains all $\binom{6}{3} = 20$ products of $3$ distinct variables from the same set.
In the holonomy equation for $\lambda$ we have $n_\lambda = 2$, Choosing a basepoint on the vertical side, $\lambda_1$ cuts off corner $1$ clockwise, which has label $d$, so $c_1^\lambda = 1$, $\delta_1^\lambda = -1$, and $m_{\Delta_1^\lambda, c_1^\lambda} = d$. Similarly, $\lambda_2$ cuts off corner $3$ anticlockwise, which has label $c$, so $c_2^\lambda = 3$, $\delta_2^\lambda = 1$ and $m_{\Delta_2^\lambda, c_2^\lambda} = c$. Thus the holonomy equation for $\lambda$ is $c-d=0$, and similarly the holonomy equation for $\mu$ is $a-e=0$.

This system of $8$ equations in $6$ variables has a unique solution in positive reals, with any one edge equation and the one vertex equation being redundant. Indeed, using the equalities $c=d$ and $a=e$ from the holonomy equations, the edge equations become
\[
a^2 = c^2, \quad
a^2 = bf, \quad
c^2 = bf,
\]
of which any two imply the third. Since all variables must be positive, these imply $a=c$. This reduces us to $3$ variables $a(=c=d=e),b,f$, and the triangle equations together with the remaining edge equation become
\[
a^2 = bf, \quad
a^2 b = 2a+b, \quad
a^2 f = 2a+f.
\]
These equations straightforwardly have unique positive solution $a=b=f=\sqrt{3}$. It is easily verified that the vertex equation is then satisfied, hence redundant.

All variables equal to $\sqrt{3}$ corresponds to all angles equal to $\pi/3$, and we have a familiar hexagonal tessellation of equilateral triangles with their Soddy circles.

\subsection{Related work}
Beyond the relations already mentioned between circle packing and work of Thurston \cite{Thurston_notes} on hyperbolic manifolds, which inspired the present paper, we mention some related work.

Much of the development in circle packing theory in recent times has been variational or transcendental in nature, such as in work of Colin de Verdi\`{e}re \cite{deVerdiereCercles}, Br\"{a}gger \cite{BraggerKreispackungen}, Rivin \cite{RivinEuclidStructure}, Leibon \cite{LeibonDelaunayTriang}, Bobenko--Springborn \cite{BobenkoVariationalPrinciples}, Stephenson \cite{StephensonKenneth2005Itcp}, and many others.
Indeed, this paper was also partly inspired by the desire to complement such variational approaches.

However, algebraic approaches have been taken in prior work in related areas; we mention a few examples.  Bobenko--Suris in \cite{BobenkoIntegrableSystems} consider algebraic relations arising from isoradial bipartite quad-graphs in $\mathbb{C}$ and their integrability. The discrete isothermic and minimal surfaces studied by Bobenko--Hoffmann--Springborn in \cite{BobenkoMinimalSurfaces} involve algebraic equations, and equations for spherical circle patterns, 
of a somewhat similar flavour. The linear and nonlinear theories of discrete analytic functions studied by Bobenko--Mercat--Suris in \cite{BobenkoLinearNonlinear} also involve algebraic equations in  a similar spirit, especially integrability conditions such as their (12).
An algebraic approach is also taken in work of Kenyon--Lam--Ramassamy--Russkikh in \cite{KenyonDimerCircle}, establishing a correspondence between circle patterns and face-weighted bipartite planar graphs, under which the cluster mutations (spider moves) of a dimer model correspond to applications of Miquel's six-circles theorem. 

\subsection{Structure of this paper}

In \refsec{preliminaries} we discuss preliminary results from classical Euclidean triangle geometry, and the algebraic results we need relating angles and cotangents. In \refsec{Descartes} we consider generalised Descartes circle theorems and prove \refthm{gen_Descartes}, using elementary and spinorial methods. In \refsec{circle_packings_solve_equations} we show that circle packings satisfy circle packing equations, and in \refsec{solutions_to_packings} we show the converse, that solutions to the circle packing equations describe circle packings, proving 
 the main \refthm{circle_packing_eqs_general}, including its spherical version \refthm{circle_packing_eqs_spherical} and unbranched version \refthm{circle_packing_eqs_general_unbranched}. Our method of proof is constructive and quite explicit, showing how certain variables and equations contribute certain geometric information. Along the way we prove the geometric interpretation of the variables of \refprop{m_meaning}, and the equivalence of reduced systems of equations in \refprop{circle_packing_reduction} and \refprop{circle_packing_reduction_torus}.

\subsection{Acknowledgements}

The authors are supported by Australian Research Council grant DP210103136. They thank Ian Agol for sharing his observations and suggestions, Kenneth Stephenson for helpful comments, and Feng Luo for useful suggestions.

\section{Preliminaries}
\label{Sec:preliminaries}

\subsection{Euclidean triangle geometry}
\label{Sec:Euclidean_triangles}

Given a Euclidean triangle $ABC$, its \emph{Soddy circles} are the unique mutually externally tangent circles centred at $A,B,C$. Denoting their radii by $r_A, r_B, r_C$ and their curvatures by $\kappa_A, \kappa_B, \kappa_C$ respectively, the side lengths of $ABC$ are $r_A + r_B$, $r_B + r_C$ and $r_C + r_A$. Let the internal angles at $A,B,C$ be $\theta_A, \theta_B, \theta_C$ respectively. See \reffig{SoddyCircleExample}. Further, let $m_A, m_B, m_C$ be the parameters given by \refeqn{ms_in_terms_of_ks}, so that 
\begin{equation}
\label{Eqn:mX_nice}
m_A = \frac{\sqrt{L_\Delta}}{\kappa_A}, \quad
m_B = \frac{\sqrt{L_\Delta}}{\kappa_B}, \quad
m_C = \frac{\sqrt{L_\Delta}}{\kappa_C}
\quad \text{where} \quad
L_\Delta = \kappa_A \kappa_B + \kappa_B \kappa_C + \kappa_C \kappa_A.
\end{equation}

\begin{figure}
 \centering
\begin{tikzpicture}[scale=1.1]
 \coordinate (A) at (0,1);
 \coordinate (B) at (-1.7,-1.1);
 \coordinate (C) at (2,-1);
 \draw[thick] (A)--(B)--(C)--(A);
 \draw[thick] (A) circle (0.91) node[above] {$A$};
 \draw[thick] (B) circle (1.78) node[left] {$B$};
 \draw[thick] (C) circle (1.92) node[right] {$C$};
 \pic [draw, ->, "$\theta_A$", angle eccentricity=1.38] {angle = B--A--C};
 \pic [draw, ->, "$\theta_B$", angle eccentricity=1.5] {angle = C--B--A};
 \pic [draw, ->, "$\theta_C$", angle eccentricity=1.5] {angle = A--C--B};
 \node[left] at (-0.15,0.85) (rx1) {$r_A$};
 \node[right] at (0.15,0.85) (rx2) {$r_A$};
 \node[below] at (-1,-1) (ry1) {$r_B$};
 \node[left] at (-1,0) (ry2) {$r_B$};
 \node[below] at (1,-1) (rz1) {$r_C$};
 \node[right] at (1,0) (rz2) {$r_C$};
\end{tikzpicture}
 \caption{A triangle and its Soddy circles.}
 \label{Fig:SoddyCircleExample}
\end{figure}

Essentially all the arguments of this paper follow from the following observation, pointed out to us by Ian Agol.
\begin{lem}
\[
\cos \theta_A = \frac{m_A^2 - 1}{m_A^2 + 1}
\]
\end{lem}
Of course, the $A$ can be replaced with $B$ or $C$.

\begin{proof}
By the cosine rule, using $\kappa_\bullet = 1/r_\bullet$, and noting
$m_A^2 \pm 1 = (L_\Delta \pm \kappa_A^2)/\kappa_A^2$, we have
\begin{align*}
\cos \theta_A &= \frac{ (r_A + r_B)^2 + (r_{A}+r_C)^2 - (r_{B}+r_C)^2}{2(r_{A}+r_B)(r_A+r_C)} 
= \frac{ r_A^2 + r_A r_B + r_A r_C - r_B r_C}{r_A^2 + r_A r_B + r_A r_C + r_B r_C} \\
&= \frac{ \kappa_{B} \kappa_C + \kappa_{A} \kappa_C + \kappa_A \kappa_B - \kappa_A^2}{\kappa_{B} \kappa_C + \kappa_{A} \kappa_C + \kappa_A \kappa_B + \kappa_A^2} 
= \frac{L_\Delta - \kappa_A^2}{L_\Delta + \kappa_A^2}
= \frac{m_A^2-1}{m_A^2+1}.
\end{align*}
\end{proof}

Various other calculations then follow straightforwardly.
\begin{lem} \
\label{Lem:exp_Agol}
\begin{enumerate}
\item 
$\sin \theta_A = \frac{2m_A}{m_A^2 + 1}$
\item 
$e^{i \theta_A} = \frac{m_A + i}{m_A - i}$
\item 
$\arg (m_A + i) = \theta_A/2$.
\end{enumerate}
\end{lem}

\begin{proof}
For (1) we calculate
\[
\sin^2 \theta_A = 1 - \cos^2 \theta_A
= 1 - \frac{ \left( m_A^2-1 \right)^2}{ \left( m_A^2+1 \right)^2}
= \frac{4m_A^2}{\left( m_A^2 + 1 \right)^2}.
\]
Since $m_A$ and $\sin \theta_A$ are both positive, taking a square root gives the desired equality. Then
\[
e^{i \theta_A} = \cos \theta_A + i \sin \theta_A 
= \frac{m_A^2-1}{m_A^2 + 1} + \frac{2 i m_A }{m_A^2 + 1} 
= \frac{ \left( m_A + i \right)^2}{\left( m_A - i \right) \left( m_A + i \right)}
= \frac{m_A + i }{m_A - i},
\]
giving (2). As $m_A + i$ and $m_A - i$ are conjugate, $\arg(m_A - i) = -\arg(m_A + i)$. The argument $\theta_A$ of $e^{i\theta_A}$ is thus $2 \arg (m_A + i)$ mod $2\pi$, and as $m_A$ is positive and $\theta_A \in (0, \pi)$, we have (3). 
\end{proof}

\begin{lem}
\label{Lem:edge_eqn_fact}
\[
\cot \left( \frac{\theta_A}{2} \right) = m_A
\quad \text{and} \quad
\frac{ \cot \left( \theta_A/2 \right)}{ \cot \left( \theta_B / 2 \right) }
= \frac{m_A}{m_B} 
= \frac{\kappa_B}{\kappa_A}
\]
\end{lem}

\begin{proof}
The first statement is immediate from \reflem{exp_Agol} (3), considering a right-angled triangle formed by $0$, $m_A$ and $m_A + i$ in the complex plane. The second statement then follows from writing $m_A, m_B$ in the form of \refeqn{mX_nice}, with the same numerator $\sqrt{L_\Delta}$.
\end{proof}

In classic Euclidean geometry notation, with the side lengths of a triangle written as $a,b,c$, angles as $A,B,C$ and semiperimeter as $s$, we have $s = r_A + r_B + r_C$, so $r_A = s-a$, $r_B = s-b$ and $r_C = s-c$. \reflem{edge_eqn_fact} then follows from classic identities, straightforwardly derived from the cosine rule and double angle expressions:
\begin{gather*}
\sin \left( \frac{A}{2} \right) = \sqrt{ \frac{(s-b)(s-c)}{bc}} = \sqrt{ \frac{r_B r_C}{bc}}, \quad
\cos \left( \frac{A}{2} \right) = \sqrt{ \frac{s(s-a)}{bc} } = \sqrt{ \frac{(r_A + r_B + r_C) r_A}{bc} }, \\
\cot \left( \frac{A}{2} \right) = \sqrt{\frac{s(s-a)}{(s-b)(s-c)}}
= \sqrt{ \frac{(r_A + r_B + r_C) r_A}{r_B r_C}}
= \frac{ \sqrt{\kappa_A \kappa_B + \kappa_B \kappa_C + \kappa_C \kappa_A}}{\kappa_A}.
\end{gather*}

\subsection{Angle and cotangent identities}
\label{Sec:angles_cotangents}

In this section we prove some elementary algebraic facts we need in the sequel. These all concern variables of the form $m_\bullet \in \R_+$ and $\theta_\bullet \in (0, \pi)$, related by $m = \cot(\theta/2)$, or equivalently, $\arg(m + i) = \theta/2$. For now we just conceive of the $m_\bullet$ and $\theta_\bullet$ as real variables, without any associated geometry. We show certain properties of the $\theta_\bullet$ are equivalent to properties of the $m_\bullet$.

\begin{lem}
\label{Lem:triangle_equation_angle_sum}
Suppose $m_1, m_2, m_3 \in \R_+$ and $\theta_1, \theta_2, \theta_3 \in (0, \pi)$ satisfy $m_j = \cot(\theta_j/2)$ for $j=1,2,3$. Then the following are equivalent.
\begin{enumerate}
\item $\theta_1 + \theta_2 + \theta_3 = \pi$.
\item $m_1 m_2 m_3 = m_1 + m_2 + m_3$.
\end{enumerate}
\end{lem}

\begin{proof}
We have $\arg (m_j + i) = \theta_j / 2$. So if (1) holds then $\arg (m_1 + i) (m_2 + i) (m_3 + i) = \pi/2$, so its real part $m_1 m_2 m_3 - m_1 - m_2 - m_3$ must be zero. Conversely, if (2) holds then the real part of $(m_1 + i)(m_2 + i)(m_3 + i)$ is zero, so its argument $(\theta_1 + \theta_2 + \theta_3)/2$ is equal to $\pi/2$ mod $\pi$. Hence $\theta_1 + \theta_2 + \theta_3$ is equal to $\pi$ mod $2\pi$. As each $\theta_j \in (0, \pi)$, we have $\theta_1 + \theta_2 + \theta_3 = \pi$.
\end{proof}

\begin{lem}
\label{Lem:vertex_equation_angle_sum}
Suppose $m_j \in \R_+$ and $\theta_j \in (0, \pi)$ satisfy $m_j = \cot(\theta_j/2)$ for $j \in \Z/d\Z$. Then the following are equivalent.
\begin{enumerate}
\item $\sum_{j=1}^d \theta_j = 2\pi (\beta + 1)$ for some integer $\beta \geq 0$.
\item $\sum_{I} (-1)^{\frac{ d-|I|-1 }{2} } \prod_{j \in I} m_{j} = 0$, where the sum is over $I \subseteq \Z/d\Z$ such that $|I| \equiv d-1$ mod $2$.
\end{enumerate}
\end{lem}

\begin{proof}
Again $\arg(m_j + i) = \theta_j / 2$. If (1) holds then $\arg \prod_{j=1}^d (m_j + i) = \pi (\beta + 1)$, so the product is real, hence its imaginary part $\sum_{I} (-1)^{\frac{ d-|I|-1 }{2} } \prod_{j \in I} m_{j}$ is zero. Conversely, if (2) holds then this imaginary part is zero, so $\prod_{j=1}^d (m_j + i)$ is nonzero and real, so its argument $\frac{1}{2} \sum_{j=1}^d \theta_j = \pi (\beta + 1)$ for some integer $\beta \geq 0$. Then $\sum_{j=1}^d \theta_j = 2\pi (\beta+1)$ as claimed.
\end{proof}

\section{Generalised Descartes circle theorems}
\label{Sec:Descartes}

In this section we discuss generalised Descartes circle theorems, including \refthm{gen_Descartes} and the results of \cite{MZ25}.

Throughout this section, let $K$ be the simplicial complex of an $n$-flower, with central vertex $v_\infty$, and boundary vertices, in cyclic order, $v_1, \ldots, v_n$, which we regard as $v_j$ over $j \in \Z/n\Z$. Let $\Delta_j$ for $j \in \Z/n\Z$ be the triangle in $K$ with vertices $v_\infty, v_{j-1}, v_j$. See \reffig{flower_complex} (left). 

In a realisation $\Phi \colon K \To \R^2$ of $K$, we have circles $C_\infty$ and $C_j$ for $j \in \Z/n\Z$ centred at $\Phi(v_\infty)$ and $\Phi(v_j)$ respectively, with curvatures $\kappa_\infty$ and $\kappa_j$ respectively. Let $\theta_j = \theta_{\Delta_j,v}$ be the angle subtended at $\Phi(v_\infty)$ by $\Phi(v_{j-1})$ and $\Phi(v_j)$. See \reffig{AgolFigure} (right). Let $m_j = m_{\Delta_j,v}$ be the corresponding circle packing variable defined by \refeqn{ms_in_terms_of_ks}. Then by \reflem{edge_eqn_fact} we have
\[
m_j = \sqrt{ \left( \frac{\kappa_j}{\kappa_\infty} + 1 \right) \left( \frac{\kappa_{j-1}}{\kappa_\infty} + 1 \right) - 1} = \cot \left( \frac{\theta_j}{2} \right).
\]

\begin{figure}
\centering
\begin{tikzpicture}[scale=2.2]
 \filldraw[black] (0,0) circle (1pt) node[above] at (0.02,0.07) {$v_{\infty}$};
 \draw[thick] (0,0)--(1,0);
 \draw[thick] (0,0)--(-1,0);
 \draw[thick] (0,0)--(0.707,0.707);
 \draw[thick] (0,0)--(0.5,-0.866);
 \draw[thick] (0,0)--(-0.5,0.866);
 \draw[thick] (0,0)--(-0.866,-0.5);
 \draw[thick] (1,0)--(0.707,0.707);
 \draw[thick] (0.707,0.707)--(-0.5,0.866);
 \draw[thick] (-0.5,0.866)--(-1,0);
 \draw[thick] (-1,0)--(-0.866,-0.5);
 \draw[thick] (0.5,-0.866)--(1,0);
 \node[above] at (0.6,0.12) {$\Delta_1$};
 \node[above] at (0.15,0.37) {$\Delta_2$};
 \node[left] at (-0.35,0.24) {$\Delta_3$};
 \node[left] at (-0.43,-0.18) {$\Delta_4$};
 \node[below] at (-0.2,-0.35) {$\cdots$};
 \node[below] at (0.5,-0.2) {$\Delta_n$};
 \node[right] at (1,0) {$v_0$};
 \node[above right] at (0.707,0.707) {$v_1$};
 \node[above] at (-0.5,0.866) {$v_2$};
 \node[left] at (-1,0) {$v_3$};
 \node[below] at (-0.866,-0.5) {$v_4$};
 \node[right] at (0.5,-0.866) {$v_{n-1}$};
\end{tikzpicture}
\begin{tikzpicture}[scale=1.3]
\draw[thick](0.47,0.01) coordinate (B) circle (1);
\draw[thick](2.16,-0.27) coordinate (C) circle (0.7);
\draw[thick](2.26,1.45) circle (0.6);
\draw[thick](1.15,1.34) circle (0.5);
\draw[thick](2.24,0.64) circle (0.21);
\draw[thick] (1.67,0.68) coordinate (A) circle (0.35);
\draw[-, thick, gray] (1.67,0.68) -- (0.47,0.01) node[below]{$C_0$};
\draw[-, thick, gray] (1.67,0.68) -- (2.16,-0.27) node[below]{$C_1$};
\draw[-, thick, gray] (1.67,0.68) -- (2.26,1.45);
\draw[-, thick, gray] (1.67,0.68) -- (1.15,1.34);
\draw[-, thick, gray] (1.67,0.68) -- (2.24,0.64);
\draw[-, thick, gray] (0.47,0.01) -- (2.16,-0.27);
\draw[-, thick, gray] (2.16,-0.27) -- (2.24,0.64);
\draw[-, thick, gray] (2.24,0.64) -- (2.26,1.45);
\draw[-, thick, gray] (2.26,1.45) -- (1.15,1.34);
\draw[-, thick, gray] (1.15,1.34) -- (0.47,0.01);
\node at (0.47,0.01)[circle,fill,inner sep=1.5pt]{};
\node at (2.16,-0.27)[circle,fill,inner sep=1.5pt]{};
\node at (2.26,1.45)[circle,fill,inner sep=1.5pt]{};
\node at (1.15,1.34)[circle,fill,inner sep=1.5pt]{};
\node at (2.24,0.64)[circle,fill,inner sep=1.5pt]{};
\pic[draw,->, "$\theta_1$"{shift=(-60:0.55)}, angle radius=0.35cm] {angle = B--A--C};
\end{tikzpicture}
\caption{Left: The complex of an $n$-flower. Right: A realisation.}
\label{Fig:AgolFigure}
\label{Fig:flower_complex}
\end{figure}

\subsection{Symmetric generalised Descartes theorem}
\label{Sec:symmetric_Descartes}

We can now immediately give a proof of \refthm{gen_Descartes}, the symmetric generalised Descartes theorem. Note that the flower may be branched: the petal circles can branch around the central circle arbitrarily.
\begin{proof}[Proof \# 1 of \refthm{gen_Descartes}]
Since $\sum_{j=1}^n \theta_j = 2\pi (\beta+1)$ for some integer $\beta \geq 0$ (the branching index) we have $\prod_{j=1}^n e^{i \theta_j} = 1$. By \reflem{exp_Agol}(2) then
\[
\prod_{j=1}^n \frac{m_j + i}{m_j - i} = 1.
\]
\end{proof}

\subsection{Spinorial approach}

We can also give a proof of \refthm{gen_Descartes} using the spinorial approach of \cite{MZ25}. We refer to that paper for details, and proceed as follows. By a dilation, we can regard $C_\infty$ as the unit circle. We invert the flower in $C_\infty$ so that the petal circles $C_j$ map to circles $\mathring{C_j}$ internally tangent to $C_\infty$. We regard the $\mathring{C_j}$ as horocycles in the conformal disc model of the hyperbolic plane $\mathbb{D}^2$. Via the Cayley map, we consider the horocycles in the upper half space model $\mathbb{U}^2$. Each horosphere can then be regarded as having a planar decoration in the sense of \cite{M_spinors_horospheres}, corresponding to a spinor $(\xi, \eta) \in \R^2$, well defined up to sign. Its centre is $\xi/\eta \in \R$ and its Euclidean diameter is $1/\eta^2$. Each circle $C_j$ thus corresponds to a horocycle $\overline{C_j}$ in $\mathbb{U}^2$, with a spinor $\alpha_j = (\xi_j, \eta_j)$, where we choose each $\eta_j$ to be positive. By a rotation of $\mathbb{D}^2$ if necessary, we can arrange that the horospheres $\overline{C_1}, \ldots, \overline{C_n}$ have centres in increasing order on the real line, i.e. $\xi_0/\eta_0 < \xi_1/\eta_1 < \cdots < \xi_{n-1}/\eta_{n-1}$.

(In \cite{MZ25}, for convenience, we also translated the $\xi_j/\eta_j$ so that $\xi_0/\eta_0 = 0$. However, such a parabolic isometry does not correspond to a Euclidean isometry of the unit disc, and so loses generality. Nonetheless, the arguments of \cite{MZ25} carry through without this convenient assumption.)

Under the (inverse) Cayley map $\mathfrak{S}^{-1} \colon \mathbb{D}^2 \To \mathbb{U}^2$, given by $z \mapsto i(z+1)/(1-z)$, we have
\[
e^{i\phi} \mapsto \frac{(e^{i\phi}+1)i}{1-e^{i\phi}} 
= \frac{ \left( e^{i\phi/2} + e^{-i\phi/2} \right) /2 }{ - \left( e^{i\phi/2}-e^{-i\phi/2} \right) / (2i) }
= -\cot(\phi/2).
\]
Thus, if a horosphere $\mathring{C}$ in $\mathbb{D}^2$ has centre $\exp(i \phi)$ and corresponds to a horosphere $\overline{C}$ in $\mathbb{U}^2$ with an associated spinor $(\xi, \eta)$, then
\begin{equation}
\label{Eqn:spinor_angle_relation}
\frac{\xi}{\eta} = - \cot \left( \frac{\phi}{2} \right).
\end{equation}

Spinors have an antisymmetric bilinear form $\{ \cdot, \cdot \}$, discussed by Penrose--Rindler in \cite{Penrose_Rindler84}. It gives the lambda length between horospheres, as discussed in \cite{M_spinors_horospheres}. We now also define another bilinear form $\langle \cdot, \cdot \rangle$.
\begin{definition}
For real spinors $\alpha = (\xi, \eta)$ and $\alpha' = (\xi', \eta')$,
\[
\{\alpha, \alpha'\} = \xi \eta' - \xi' \eta
\quad \text{and} \quad
\langle \alpha, \alpha' \rangle = \xi \xi' + \eta \eta'.
\]
\end{definition}
The form $\langle \cdot, \cdot \rangle$ is clearly symmetric; a similar form on spinors was considered in \cite{M_Varsha_quaternionic}. We can give it a geometric interpretation.
\begin{lem}
\label{Lem:angle_from_inner_product}
Suppose horocycles $\mathring{C}, \mathring{C'}$ of $\mathbb{D}^2$ have centres at $\exp(i\phi), \exp(i\phi')$ respectively, corresponding under the Cayley transform to horocycles in $\mathbb{U}^2$ with spinors $\alpha = (\xi, \eta)$, $\alpha' = (\xi', \eta')$. Then
\[
\frac{\langle \alpha, \alpha' \rangle}{ \{ \alpha, \alpha' \} } = -\cot \left( \frac{\phi' - \phi}{2} \right).
\]
\end{lem}
In other words, $\langle \alpha, \alpha' \rangle/\{\alpha, \alpha'\}$ describes the angle subtended by the centres of $\mathring{C}, \mathring{C'}$ at the centre of $\mathbb{D}^2$; this angle is also the angle subtended by the centres of $\overline{C}, \overline{C'}$ at $i$ in $\mathbb{U}^2$.
\begin{proof}
The cotangent addition formula and \refeqn{spinor_angle_relation} yield
\[
\cot \left( \frac{\phi' - \phi}{2} \right)
= \frac{ -\cot (\phi/2) \cot (\phi'/2) -1 }{\cot (\phi'/2) - \cot (\phi/2) }
= \frac{ - \frac{\xi \xi'}{\eta \eta'} - 1}{- \frac{\xi'}{\eta'} + \frac{\xi}{\eta} }
= \frac{-\xi \xi' - \eta \eta'}{\xi \eta' - \xi' \eta}
= \frac{-\langle \alpha, \alpha' \rangle}{ \{ \alpha, \alpha' \} }.
\]
\end{proof}

Returning to the approach of \cite{MZ25}, as consecutive horospheres $C_{j-1}, C_j$ are tangent we have $\{\alpha_{j-1}, \alpha_j\} = \pm 1$. With our choice of all $\eta_j$ positive we in fact have \cite[lemma 4.2]{MZ25}
\[
\{\alpha_0, \alpha_1 \} = \cdots = \{\alpha_{n-2}, \alpha_{n-1}\} = -1, \quad
\{ \alpha_0, \alpha_{n-1} \} = -1.
\]
Letting the centre of the horosphere $\mathring{C_j}$ be $\exp(i \phi_j)$, as we have arranged $\overline{C_0}, \overline{C_1}, \ldots, \overline{C_{n-1}}$ to have centres in increasing order along $\R$, then we may take $0 < \phi_0 < \phi_1 < \cdots < \phi_{n-1} < 2\pi$. See \reffig{arguments_angles}. Now the angle $\theta_j$ subtended at $\Phi(v_\infty)$ by $\Phi(v_{j-1})$ and $\Phi(v_j)$ in the $n$-flower is preserved under inversion, to become the angle subtended at $0$ in $\mathbb{D}^2$ by the centres (as Euclidean circles or hyperbolic horocycles) of $\mathring{C}_{j-1}$ and $\mathring{C_j}$, so $\theta_j = \phi_j - \phi_{j-1}$. This is true generally only mod $2\pi$; it is true on the nose for $j=1, \ldots, n-1$, and we have $\theta_0 = \phi_0 - \phi_{n-1} + 2\pi$. Being an angle in a tangent triple of circles (or in a Euclidean triangle), each $\theta_j$ satisfies $0 < \theta_j < \pi$.

\begin{figure}
\begin{center}
\begin{tikzpicture}
 \draw[thick] (0,0) circle (2);
 \coordinate (O) at (0,0);
 \coordinate (A) at (2,0);
 \coordinate (B) at (1.732,1);
 \coordinate (C) at (1,1.732);
 \coordinate (D) at (-1.147,1.638);
 \coordinate (E) at (-1.970,0.347);
 \coordinate (F) at (1.414,-1.414);
 \coordinate (start) at (-1.879,-0.684);
 \coordinate (end) at (0.517,-1.93);
 \draw[dashed,gray] (0,0)--(A);
 \draw[thick] (0,0)--(1.732,1);
 \draw[thick] (0,0)--(1,1.732);
 \draw[thick] (0,0)--(-1.147,1.638);
 \draw[thick] (0,0)--(-1.970,0.347);
 \draw[thick] (0,0)--(1.414,-1.414);
 \pic [draw, ->, "$\theta_0$", angle eccentricity=1.5] {angle = F--O--B};
 \pic [draw, ->, "$\theta_1$", angle eccentricity=1.5] {angle = B--O--C};
 \pic [draw, ->, "$\theta_2$", angle eccentricity=1.5] {angle = C--O--D};
 \pic [draw, ->, "$\theta_3$", angle eccentricity=1.5] {angle = D--O--E};
 \pic [draw, thick, dotted, angle radius=1cm, angle eccentricity=1.5] {angle = start--O--end};
 \node at (2*1.15,0*1.15) {$0$};
 \node at (1.732*1.15,1*1.15) {$\phi_0$};
 \node at (1*1.15,1.732*1.15) {$\phi_1$};
 \node at (-1.147*1.15,1.638*1.15) {$\phi_2$};
 \node at (-1.970*1.15,0.347*1.15) {$\phi_3$};
 \node at (1.414*1.15,-1.414*1.15) {$\phi_{n-1}$};
\end{tikzpicture}
\end{center}
 \caption{Arguments $\phi_j$ of centres of horospheres, and angles $\theta_j$.}
 \label{Fig:arguments_angles}
\end{figure}

Now \reflem{angle_from_inner_product} applied to $\alpha_{j-1}, \alpha_j$ for $1 \leq j \leq n-1$ gives
\[
m_j
= \cot \left( \frac{\theta_j}{2} \right) 
= \cot \left( \frac{\phi_j - \phi_{j-1}}{2} \right)
= \langle \alpha_{j-1}, \alpha_j \rangle.
\]
since $\{\alpha_{j-1}, \alpha_j\} = -1$. Applied to $\alpha_{n-1}, \alpha_0$, using $\phi_0 - \phi_{n-1} = \theta_0 - 2\pi$ and $\{\alpha_{n-1}, \alpha_0\} = 1$, we obtain
\[
m_0
= \cot \left( \frac{\theta_0}{2} \right)
= \cot \left( \frac{\theta_0 - 2\pi}{2} \right) = \cot \left( \frac{\phi_0 - \phi_{n-1}}{2} \right)
= - \langle \alpha_{n-1}, \alpha_0 \rangle.
\]

Letting $z_j = \xi_j + i \eta_j$, we then have
\[
z_{j-1} \, \overline{z_j}
= \langle \alpha_{j-1}, \alpha_j \rangle - i \{ \alpha_{j-1}, \alpha_j \}.
\]
Thus
\[
z_{j-1} \, \overline{z_j}
= m_j + i \quad \text{for $1 \leq j \leq n-1$,}
\quad \text{and} \quad
z_{n-1} \, \overline{z_0}
= - m_0 - i.
\]
(In \cite{MZ25} we only considered $z_{j-1} \overline{z_j}$ for $1 \leq j \leq n-1$.)

Thus as discussed in \cite[section 5]{MZ25}, the sequence $z_0, z_1, \ldots, z_{n-1}$ lies in the upper half of the complex plane, with arguments strictly decreasing, and each pair $z_{j-1}, z_j$ (and $z_{n-1}, z_0$) spanning a parallelogram of area $1$. (The assumption $\xi_0/\eta_0 = 0$ made $z_0$ lie on the imaginary axis, but the reasoning works without this assumption.) By multiplying together the $z_{j-1} \overline{z_j}$ for $1 \leq j \leq n-1$, we obtain $\prod_{j=1}^{n-1} (m_j + i)$; doing the same for their conjugates, and subtracting, we obtain the generalised Descartes formula of \cite{MZ25}.

If instead we multiply together all $n$ of the terms $z_{j-1} \overline{z_j}$, we obtain a ``spinorial" proof of the symmetric generalised Descartes equation of \refthm{gen_Descartes}.

\begin{proof}[Proof \# 2 of \refthm{gen_Descartes}]
\[
\prod_{j=1}^n |z_j|^2 = \prod_{j=1}^{n} z_{j-1} \, \overline{z_j}
= (m_1 + i)(m_2 + i) \cdots (m_{n-1} + i)(-m_0 -i)
= - \prod_{j=1}^n \left( m_j + i \right).
\]
Thus
\[
\prod_{j=1}^n \left( m_j + i \right) = - \prod_{j=1}^n |z_j|^2,
\]
and the conjugate of this calculation shows that $\prod_{j=1}^n \left( m_j - i \right) = - \prod_{j=1}^n |z_j|^2$ also, 
implying the result.
\end{proof}

\section{Circle packing equations}
\label{Sec:circle_packings_solve_equations}

In this section we show that a circle packing of the type considered in \refthm{circle_packing_eqs_general} obeys the circle packing equations. We will first deal with the case where $K$ is planar, i.e. an open, closed or semi-open disc; we refer to all these cases as discs for convenience. We then deal with spheres and tori in \refsec{spherical_case_eqns} and \refsec{torus_eqns_satisfied} respectively.

\subsection{ Disc case}
\label{Sec:disc_case}

Let $K$ be a simplicial complex triangulating an oriented plane surface. We may have $K$ finite or infinite, triangulating a closed disc, or open or semi-open disc respectively. 
As $K$ triangulates a surface, each interior vertex has finite degree. However boundary vertices may have infinite degree. Infinite-degree boundary vertices have neighbourhoods homeomorphic to an open half-disc or half-open half-disc, i.e. to a neighbourhood of $0$ in $\{z \in \C \mid z = 0 \text{ or } \arg z \in (0, \pi) \}$ or $\{z \in \C \mid z = 0 \text{ or } \arg z \in [0, \pi) \}$ respectively. The latter case is slightly more general than the usual notion of surface with boundary.

Suppose we have a circle packing for $K$ in $\R^2$, and let $\Phi$ be a realisation, with circles $\{C_v\}_{v \in V(K)}$, radii $\{r_v\}_{v \in V(K)}$, curvatures $\{\kappa_v\}_{v \in V(K)}$, and angles $\{\theta_{\Delta,v}\}_{(\Delta,v) \in C(K)}$. Note that all $\kappa_v, r_v, \theta_{\Delta,v}$ are positive and moreover $0 < \theta_{\Delta,v} < \pi$. (We can denote corners by $(\Delta,v)$, where $\Delta \in F(K)$ and $v$ is a vertex of $\Delta$, since $K$ is simplicial.) Note that when $K$ is infinite, the sets $V(K), C(K)$ are also infinite, and we have infinitely many quantities $\kappa_v, r_v, \theta_{\Delta,v}$.

To each corner $(\Delta, v)$ of a triangle $\Delta$ of $K$ with vertices $u,v,w$, we also associate a positive real number $m_{\Delta,v}$ by \refeqn{ms_in_terms_of_ks}, i.e. $m_{\Delta,v} = \sqrt{L_\Delta}/\kappa_v$, where $L_\Delta = \kappa_u \kappa_v + \kappa_v \kappa_w + \kappa_w \kappa_u$. Note that $m_{\Delta,v}$ and $L_\Delta$ are manifestly positive real numbers. By \reflem{edge_eqn_fact} and \reflem{exp_Agol} then $m_{\Delta,v} = \cot ( \theta_{\Delta,v} / 2)$ and $\arg (m_{\Delta,v} + i) = \theta_{\Delta,v}/2$. 

 To verify the triangle equations, consider
a triangle $\Delta$ of $K$ with vertices $u,v,w,$ as in \reffig{circle_packing_eqn_vars} (right). We have three angles $\theta_{\Delta,u}$, $\theta_{\Delta,v}$ and $\theta_{\Delta,w}$, which we abbreviate to $\theta_u, \theta_v, \theta_w$; being the angles of a Euclidean triangle we have $\theta_u + \theta_v + \theta_w = \pi$. Similarly, we have $m_{\Delta, u}, m_{\Delta,v}, m_{\Delta,w}$ which we abbreviate to $m_u, m_v, m_w$.

\begin{lem}[Triangle equations]
\label{Lem:triangle_eqns}
The three quantities $m_u, m_v, m_w$ satisfy
\[
\label{Eqn:triangle_equation}
m_u m_v m_w = m_u + m_v + m_w.
\]
\end{lem}

\begin{proof}
By \reflem{triangle_equation_angle_sum} this follows immediately from $\theta_u + \theta_v + \theta_w = \pi$.
\end{proof}

 To verify the edge equations,
suppose we have an interior edge $e$ of $K$, joining vertices $v_1,v_2$, with a triangle $\Delta_1$ to one side and a triangle $\Delta_2$ to the other, as in \reffig{circle_packing_eqn_vars} (centre). We consider the four corners $(\Delta_j, v_k)$ over $j,k=1,2$. 
\begin{lem}[Edge equations]
\label{Lem:edge_eqns}
The four quantities $m_{\Delta_1,v_1}, m_{\Delta_1,v_2}, m_{\Delta_2, v_1}, m_{\Delta_2, v_2}$ satisfy
\begin{equation}
\label{Eqn:edge_eqn}
m_{\Delta_1, v_1} m_{\Delta_2, v_2} = m_{\Delta_1, v_2} m_{\Delta_2, v_1}.
\end{equation}
\end{lem}

\begin{proof}
This follows immediately from, for instance, \reflem{edge_eqn_fact}:
\[
\frac{m_{\Delta_1,v_1}}{m_{\Delta_1,v_2}} = \frac{\kappa_{v_2}}{\kappa_{v_1}}
= \frac{m_{\Delta_2,v_1}}{m_{\Delta_2,v_2}}.
\]
\end{proof}

Finally, to verify the vertex equations, consider an interior vertex $v$ of $K$ of degree $d$, with adjacent triangles $\Delta_j$ in anticlockwise cyclic order over $j \in \Z/d\Z$, as in \reffig{circle_packing_eqn_vars} (left). We have angles $\theta_{\Delta_j,v}$ over $j \in \Z/d\Z$, which we abbreviate to $\theta_j$; similarly we abbreviate $m_{\Delta_j,v}$ to $m_v$. Let the branching index of $\Phi$ at $v$ be $\beta$, so that $\beta \geq 0$ is an integer and $\sum_{j} \theta_j = 2 \pi (\beta + 1)$.
\begin{lem}[Vertex equations]
\label{Lem:vertex_eqns}
The quantities $m_j$ over $j \in \Z/d\Z$ satisfy
\[
\sum_{I} (-1)^{\frac{ d-|I|-1 }{2} } \prod_{j \in I} m_{j} = 0,
\]
summing over $I \subseteq \Z/d\Z$ such that $|I| \equiv d-1$ mod $2$. Moreover if $\beta = 0$ then
\[
\label{Eqn:vertex_eqn_arg_version}
\sum_{j=1}^d \arg \left( m_{j} + i \right) = \pi.
\]
\end{lem}

\begin{proof}
By \reflem{vertex_equation_angle_sum} the first equation follows immediately from $\sum_j \theta_j = 2\pi (\beta+1)$. The second equation follows immediately from $\arg (m_j + i) = \theta_j / 2$.
\end{proof}

It is immediate now from \reflem{triangle_eqns}, \reflem{edge_eqns} and \reflem{vertex_eqns} that, when $K$ is a simplicial complex triangulating a disc, a circle packing of $K$ in the Euclidean plane yields numbers $\{m_{\Delta,v}\}_{(\Delta,v) \in C(K)}$ satisfying the vertex, edge and triangle equations, and an unbranched circle packing of $K$ satisfies the unbranched vertex equations. When $K$ is infinite, there are infinitely may $m_{\Delta,v}$ satisfying infinitely many equations. Moreover, if two circle packings of $K$ are conformally equivalent, then they are related by an overall similarity of $\R^2$, hence yield the same numbers $m_\bullet$ satisfying the circle packing equations. In other words, we have proved one direction of \refthm{circle_packing_eqs_general} and \refthm{circle_packing_eqs_general_unbranched}, with the relations of \refprop{m_meaning} holding, when $K$ is a disc. 

 The above argument applies equally well when in slightly more general circumstances. If we consider a circle packing with some triples of circles reversing orientation, so that $\Phi$ may fold along some edges, the above arguments still yield angles $\theta_{\Delta,v}$, where we take negative angles in triangles with reversed orientation, and quantities $m_{\Delta,v}$ in the same way (which may now be negative). 

The above argument also applies to circle packings in the Euclidean plane of complexes $K$ which are not simply connected. A similar construction yields a realisation $\Phi$, angles $\theta_\bullet$ and variables $m_\bullet$ satisfying the circle packing equations. However, the converse is not true in general; we cannot rebuild a circle packing from a solution to the circle packing equations. See the comment at the end of \refsec{closed_disc}.

\subsection{Spherical case}
\label{Sec:spherical_case_eqns}

 Now let $K$ be a simplicial complex triangulating an oriented sphere. As in \refsec{circle_packing_eqns} let $\Delta_0$ be a triangle of $K$, and let $K_0$ be the simplicial 2-complex triangulating an oriented closed disc obtained by removing the interior of $\Delta_0$ from $K$. As $K$ is a simplicial complex, we may denote its corners by $(\Delta,v)$ where $\Delta \in F(K)$ and $v \in V(K)$ is a vertex of $\Delta$. Note that each corner of $K$ is either a corner of $K_0$ or a corner of $\Delta_0$. 

Let the vertices of $\Delta_0$ in order around its oriented boundary be $v_0, v_1, v_2$. These vertices are also the boundary vertices of $K_0$. Around the oriented boundary $\partial K_0$, the vertices are in the opposite order $v_2,v_1,v_0$.

Let $\Phi \colon K \To S^2$ be the realisation map of a circle packing for $K$ which is good for $\Delta_0$, with circles $C_v$ for each $v \in V(K)$. Then the north pole $N$ lies in the interstice of $\Delta_0$, and $\Phi^{-1}(N)$ is a single interior point of $\Delta_0$. Moreover, $\Phi(\Delta_0)$ is a geodesic triangle on $S^2$ containing $N$ in its interior, and its oriented boundary has vertices $\Phi(v_0), \Phi(v_1), \Phi(v_2)$ in order. If we regard each circle $C_v$ of the packing as bounding a disc containing $\Phi(v)$, then all the discs of the packing are disjoint from $N$. Restricting to $K_0$, we observe that $\Phi|_{K_0}$ is the realisation map of a circle packing of $K_0$ on $S^2$, whose image avoids $N$. 

We now apply a conformal automorphism of $S^2$ (i.e. M\"{o}bius transformation) to this spherical circle packing. For $j \in \Z/3\Z$, let $C_j$ be a circle on $S^2$ such that each $C_j$ and $C_{j+1}$ are tangent, and such that $C_{j+1}$ is obtained from $C_j$ by $2\pi/3$ rotation about the oriented axis from south to north pole. In other words, the $C_j$ form three mutually tangent circles arranged symmetrically about $N$. M\"{o}bius transformations act simply transitively on triples of tangent circles, so there is a unique M\"{o}bius transformation $M$ such that $M(C_{v_j}) = C_j$. As $\Phi(v_0), \Phi(v_1), \Phi(v_2)$ encircle $N$ around the oriented boundary of $\Phi(\Delta_0)$, $M$ sends the interstice between $C_{v_0}, C_{v_1}, C_{v_2}$ containing $N$ to the interstice between $C_0, C_1, C_2$ containing $N$. 

Applying $M$ to our spherical circle packing, we obtain another spherical circle packing of $K$ good for $\Delta_0$, with realisation $\Phi_M$. (Note $\Phi_M \neq M \circ \Phi$ in general since M\"{o}bius transformations do not generally preserve circle centres.) In this circle packing, $\Phi_M (\Delta_0)$ is a geodesic triangle which is symmetric under $2\pi/3$ rotation about the north-south axis, and the circle corresponding to each $v_j$ is $C_j$. Each circle of this packing still bounds a disc disjoint from $N$, similarly to the original packing. Restricting to $K_0$, we observe $\Phi_M|_{K_0}$ is the realisation of a circle packing of $K_0$ on $S^2$ whose image avoids $N$.

We may then apply stereographic projection $\mathcal{S} \colon S^2 \setminus \{N\} \To \R^2$ to this circle packing of $K_0$ on $S^2$. We obtain a circle packing of $K_0$ in the Euclidean plane. With our chosen conventions, $\mathcal{S}$ is orientation reversing, so this is really a circle packing of $K_0$ with reversed orientation, which we denote $\overline{K_0}$. 

Let this circle packing have realisation $\overline{\Phi} \colon \overline{K_0} \To \R^2$, with circles $\{\overline{C_v}\}_{v \in V(\overline{K_0})}$, radii $\{r_v\}_{v \in V(\overline{K_0})}$, curvatures $\{\kappa_v\}_{v \in V(\overline{K_0})}$, and angles $\{\theta_{\Delta,v}\}_{(\Delta,v) \in C(\overline{K_0})}$. (Note $\overline{\Phi} \neq \mathcal{S} \circ \Phi_M|_{K_0}$ in general, since stereographic projection does not generally preserve circle centres.) The vertices of $\Delta_0$ are also vertices of $K_0$, so $\overline{C_{v_j}}, r_{v_j}, \kappa_{v_j}$ are all well-defined for the vertices $v_j$ of $\Delta_0$. 

As the circles $C_j$ on $S^2$ of the packing $\Phi_M$ are related by $2\pi/3$ rotation of $S^2$ about the south-to-north axis, their images $\overline{C_{v_j}}$ under $\mathcal{S}$ are related by $2\pi/3$ rotation of $\R^2$ about the origin. So all $r_{v_j}$ are equal, as are all $\kappa_{v_j}$. The points $\overline{\Phi} (v_0), \overline{\Phi} (v_1), \overline{\Phi} (v_2)$ form the vertices of an equilateral triangle in $\R^2$ centred at the origin, in anticlockwise order. These are the boundary vertices of $\overline{K_0}$, in order, so the image of $\overline{\Phi}$ consists precisely of this triangle and its interior.

We can then regard $\Delta_0$ as corresponding to the exterior of this triangle. So for the three corners $(\Delta_0, v_j)$ of $\Delta_0$, we regard them the angles $\theta_{\Delta_0, v_j}$ as the exterior angles of this triangle, and set each $\theta_{\Delta_0, v_j} = 5\pi/3$. 

As in the disc case, for each corner $(\Delta, v)$ of $\overline{K_0}$ we associate a positive real number $m_{\Delta,v}$ satisfying $m_{\Delta,v} = \cot \left( \theta_{\Delta,v}/2 \right)$. We may use this equation also to define $m_{\Delta_0,v_j}$ for the corners $(\Delta_0,v_j)$ of $\Delta_0$, setting $m_{\Delta_0,v_j} = \cot \left( 5\pi/6 \right) = - \sqrt{3}$. 
Thus, from a spherical circle packing of $K$ which is good for $\Delta_0$, we obtain numbers $m_{\Delta,v}$ over all corners $(\Delta,v)$ of $K$. Although there were choices involved in this process, the resulting numbers $m_{\Delta,v}$ do not depend on these choices. In fact we have the following.

\begin{lem}
The numbers $\{m_{\Delta,v}\}_{(\Delta, v) \in C(K)}$ depend only on the conformal class of the original circle packing in $S^2$ good for $\Delta_0$.

\end{lem}

\begin{proof}
Suppose $\Phi, \Phi'$ are the realisation maps of two conformally equivalent circle packings of $K$, both good for $\Delta_0$. Then $\Phi' = \Psi \circ \Phi$, where $\Psi$ is a M\"{o}bius transformation of the sphere. As both circle packings are good for $\Delta_0$, the north pole lies in the interstice of $\Delta_0$ in both packings. After applying a M\"{o}bius transformation to send the three circles of $\Delta_0$ to the three circles $C_j$, the packings are related by M\"{o}bius transformation which fixes all the $C_j$, hence must be the identity; thus the two packings become identical.

Any two choices of the circles $C_j$ on $S^2$ are related by a loxodromic M\"{o}bius transformation fixing the north and south poles. Similarly, any two choices of $M$, and the packing $\Phi_M$, are related by such a loxodromic. Hence any two resulting Euclidean circle packings $\overline{\Phi}$ on $\R^2$ are related by a spiral symmetry about the origin, and have the same angles $\theta_{\Delta,v}$ and numbers $m_{\Delta,v}$.
\end{proof}

\begin{lem}
The $\{m_{\Delta,v}\}_{(\Delta, v) \in C(K)}$ satisfy the circle packing equations for $(K, \Delta_0)$. If the packing is unbranched, then they satisfy the unbranched circle packing equations for $(K, \Delta_0)$.
\end{lem}

\begin{proof}
Applying the disc case to the packing of $\overline{K_0}$ with realisation $\Phi_0$, the $m_{\Delta,v}$ over $(\Delta,v) \in C(\overline{K_0})$ satisfy the circle packing equations for $\overline{K_0}$ (and $K_0$: the equations are independent of the choice of orientation of the complex). By construction, the sphere-closing equations \refeqn{sphere-closing_eqn_list} are also satisfied. Substituting all variables as $-\sqrt{3}$ into the triangle equation \refeqn{triangle_eqn_list} for $\Delta_0$, this equation is also satisfied. It thus remains to show that the $m_{\Delta,v}$ satisfy the edge equations at the edges of $\Delta_0$ and the vertex equations at the vertices of $\Delta_0$.

For $j \in \Z/3\Z$, let $e_j$ be the edge of $K$ from $v_j$ to $v_{j+1}$. So $e_j$ is an edge of $\Delta_0$ and a boundary edge of $K_0$. Let $\Delta_j'$ be the triangle of $K_0$ containing $e_j$. We have $m_{\Delta_0,v_j} = m_{\Delta_0,v_{j+1}} = -\sqrt{3}$ and $\kappa_{v_j} = \kappa_{v_{j+1}}$. By \reflem{edge_eqn_fact} in triangle $\Delta'_j$ we have
\[
\frac{m_{\Delta'_j,v_j}}{m_{\Delta'_j,v_{j+1}}}
=\frac{\kappa_{v_{j+1}}}{\kappa_{v_j}} = 1,
\quad \text{so} \quad
(-\sqrt{3})m_{\Delta'_j,v_j} = (-\sqrt{3})m_{\Delta'_j,v_{j+1}},
\]
giving the edge equation at $e_j$.

Finally, consider a vertex $v_j$, where $j \in \Z/3\Z$. As $\Phi_0$ has image an equilateral triangle with $v_j$ as a vertex, the angles $\theta_{\Delta,v_j}$ at corners of $K_0$ at $v_j$ sum to $\pi/3$. We have $\theta_{\Delta_0,v_j} = 5\pi/3$. So summing over triangles $\Delta$ with a corner at $v_j$, we have $\sum_{\Delta} \theta_{\Delta,v_j} = 2\pi$. Hence by \reflem{exp_Agol}(3) $\prod_\Delta (m_{\Delta,v_j} + i)$ has argument $\pi$, so its imaginary part is $0$, giving the vertex equation at $v_j$.

If the original circle packing on the sphere is unbranched (in fact the condition for the packing to be good implies $\Phi$ has degree $1$, so it is unbranched), then as in \refsec{disc_case}, the unbranched vertex equations are satisfied at interior vertices of $K_0$. We have observed explicitly that the angles $\theta_{\Delta,v_j}$ at each vertex of $\Delta_0$ sum to $2\pi$, so by \reflem{exp_Agol}(3) the $\arg(m_{\Delta,v_j}+i)$ sum to $\pi$, giving the remaining unbranched vertex equations. 
\end{proof}

Even though \refprop{m_meaning} is only stated for Euclidean circle packings, the relations of that proposition apply to the circle packing $\overline{\Phi}$ obtained after stereographic projection.

As in the disc case, the arguments of this section apply equally well to slightly more general circle packings. Again, if some triples of circles reverse orientation, realisation maps may fold along some edges, but using negative angles, the circle packing equations will still be satisfied. Even if a circle packing is not good, as long as no circle passes through the north pole, we may convert it to a circle packing in the plane via stereographic projection; circles ``covering the north pole" will be ``turned inside out", but using negative angles we can obtain a solution to the circle packing equations.

\subsection{Toroidal case}
\label{Sec:torus_eqns_satisfied}

Let $K$ be a $\Delta$-complex triangulating an oriented torus, with universal cover $\widetilde{K}$ a simplicial complex triangulating $\R^2$.

On $K$, there exists a sequence of oriented edges (possibly a single edge) forming a simple closed non-contractible oriented curve, which we denote $\mathfrak{l}$. Cutting along $\mathfrak{l}$ yields a $\Delta$-complex triangulating an oriented annulus. On this annulus, there exists a sequence of edges (again possibly a single edge) forming a simple properly embedded arc from one boundary component to the other, which we denote $\mathfrak{m}$. Cutting along $\mathfrak{m}$ yields a simplicial complex $K_0$ triangulating an oriented disc, whose boundary consists of four polygonal paths, namely two copies of $\mathfrak{l}$ (oriented oppositely around $\partial K_0$), and two copies of $\mathfrak{m}$ (also oriented oppositely around $\partial K_0$). This $K_0$ forms a fundamental domain for $K$: the universal cover $\widetilde{K}$ of $K$ is formed by gluing together a $\Z^2$ family of copies of $K_0$ along copies of $\mathfrak{l}$ and $\mathfrak{m}$. The curves $\mathfrak{l}$ and $\mathfrak{m}$ can be regarded as simple closed curves on $K$ which form a basis for $H_1 (K) \cong \Z^2$. By orienting $\mathfrak{l},\mathfrak{m}$ appropriately, we may assume $\mathfrak{l},\mathfrak{m}$ form a positively oriented basis, as in \reffig{fundamental_domain_complex}. As a subcomplex of the simplicial complex $\widetilde{K}$, $K_0$ is also a simplicial complex. 

\begin{figure}
\begin{center}
\begin{tikzpicture}[scale=1.2]
 \draw (0,0) -- (1,1)--(0.4,1.6)--(1.1,2.3)--(0.6,3)--(1,3.5);
 \draw (3.5,0.1) -- (4.5,1.1)--(3.9,1.7)--(4.6,2.4)--(4.1,3.1)--(4.5,3.6);
 \draw (1,3.5) -- (1.6,3.1)--(2.1,3.5)--(2.9,3.1)--(4.5,3.6);
 \draw (0,0) -- (0.6,-0.4)--(1.1,0)--(1.9,-0.4)--(3.5,0.1);
 \filldraw (0,0) circle (1pt);
 \filldraw (1,3.5) circle (1pt);
 \filldraw (3.5,0.1) circle (1pt);
 \filldraw (4.5,3.6) circle (1pt);
 \draw[->] (0.4,1.6)--(0.85,2.05);
 \draw[->] (3.9,1.7)--(4.35,2.15);
 \draw[->>] (2.1,3.5)--(2.5,3.3);
 \draw[->>] (1.1,0)--(1.5,-0.2);
 \draw[dashed,->] (-0.5,0) -- (0.5,3.5) node[midway, left] {$\mathfrak{m}$};
 \draw[dashed,->] (4.5,0.1) -- (5.5,3.6) node[midway, right] {$\mathfrak{m}$};
 \draw[dashed,->] (1,3.9) -- (4.5,4) node[midway,above] {$\mathfrak{l}$};
 \draw[dashed,->] (0,-0.7) -- (3.5,-0.6) node[midway,below ] {$\mathfrak{l}$};
 \node at (2.2,1.8) {$K_0$};
 \draw[thick, OliveGreen,->] (0.5,0.5) -- (2,0.55) node[above] {$\lambda$};
 \draw[thick, OliveGreen] (2,0.55) -- (4,0.6);
 \draw[thick,OliveGreen,->] (2.7,-0.15) -- (3.2,1.6) node[left] {$\mu$};
 \draw[thick,OliveGreen] (3.7,3.35) -- (3.2,1.6);
\end{tikzpicture}
\end{center}
 \caption{ Fundamental domain complex $K_0$, with sides consisting of two
		 copies of $\mathfrak{l}$ and $\mathfrak{m}$, and their pushoff normal curves $\lambda$ and $\mu$.}
 \label{Fig:fundamental_domain_complex}
\end{figure}

We refer to a complex $K_0$ constructed in this way as a \emph{fundamental domain complex for $K$}. Note that not all (oriented) bases of $H_1 (K)$ can arise as the (homology classes of the) boundary curves $\mathfrak{l}$ and $\mathfrak{m}$ of a fundamental domain complex for $K$.

Suppose now we have a circle packing of $K$ on a Euclidean torus $T^2$, with realisation map $\Phi \colon K \To T^2$. This circle packing develops to a circle packing of $\widetilde{K}$ in the Euclidean plane $\R^2$, with realisation map $\widetilde{\Phi} \colon \widetilde{K} \To \R^2$. This circle packing of $\widetilde{K}$ in the plane consists of copies of circle packings of the fundamental domain complex $K_0$, related by Euclidean translations. Choose one of these packings of $K_0$ arbitrarily, so we have a circle packing of $K_0$, with realisation map $\Phi_0 = \widetilde{\Phi}|_{K_0} \colon K_0 \To \R^2$. The arguments of \refsec{disc_case} above can be applied to both $K_0$ (a closed disc) and $\widetilde{K}$ (a plane / open disc). 

The corners of $K$ are naturally in bijection with corners of $K_0$, and with equivalence classes of corners of $\widetilde{K}$ under deck transformations. As $T^2$ is a Euclidean torus, corners of $\widetilde{K}$ related by deck transformations are related under $\widetilde{\Phi}$ by Euclidean translations. The angles arising in $\Phi$, $\widetilde{\Phi}$ and $\Phi_0$ are all equal at corresponding corners. 

As $K_0$ and $\widetilde{K}$ are simplicial complexes, we can denote corners of these complexes by $(\Delta, v)$ where $\Delta$ is a face and $v$ one of its vertices. We then have angles $\{\theta_{\Delta,v}\}$ in $(0, \pi)$ and positive numbers $\{m_{\Delta,v}\}$ over the corners of $K_0$ or $\widetilde{K}$, such that each $m_{\Delta,v} = \cot ( \theta_{\Delta,v}/2)$ and $\arg(m_{\Delta,v}+i) = \theta_{\Delta,v}/2$.

Thus, from a circle packing of $K$ on $T^2$, we obtain numbers $\{m_{\Delta,v}\}_{(\Delta, v) \in C(K)}$, which can be regarded as arising from $\widetilde{\Phi}$ or $\Phi_0$.

\begin{lem}
The numbers $\{m_{\Delta,v}\}_{(\Delta,v) \in C(K)}$ depend only on the conformal class of the original circle packing in $T^2$.
\end{lem}

\begin{proof}
Suppose $\Phi, \Phi'$ are the realisation maps of two conformally equivalent toroidal circle packings of $K$. Then their developments to $\R^2$ have realisation maps $\widetilde{\Phi}, \widetilde{\Phi'} \colon \widetilde{K} \To \R^2$ which are related by a Euclidean similarity, hence have the same angles $\theta_\bullet$ in their corners.

The angles $\theta_\bullet$ of a realisation of $K$ are the same as those of a realisation of $\widetilde{K}$, hence do not depend on any of the auxiliary choices of $\mathfrak{l}, \mathfrak{m}$ or $K_0$. Thus the $m_\bullet = \cot (\theta_\bullet/2)$ do not depend on any choices.
\end{proof}

Although $\mathfrak{l}, \mathfrak{m}$ can be regarded as curves forming a basis for $H_1 (K)$, they are not normal curves (as in \refsec{circle_packing_eqns}) to which the holonomy equations apply. However, applying a small oriented pushoff, as in \reffig{fundamental_domain_complex}, we obtain smooth oriented simple closed curves $\lambda, \mu$ respectively, which are normal with respect to $K$. The homology classes of $\lambda, \mu$ still form an oriented basis of $H_1 (K)$, and the holonomy equations apply to them. More generally, the holonomy equations are satisfied for \emph{any} oriented closed normal curves $\lambda, \mu$ forming a basis for $H_1 (K)$ (as in \refdef{circle_packing_eqns_torus}), as we now see. 

\begin{lem}
\label{Lem:torus_packings_give_solutions}
Let $\lambda, \mu$ be oriented closed normal curves forming a basis for $H_1 (K)$. Then the $\{m_{\Delta,v}\}_{(\Delta,v) \in C(K)}$ satisfy the circle packing equations for $(K, \lambda, \mu)$. If the packing is unbranched, then they satisfy the unbranched circle packing equations for $(K, \lambda, \mu)$.
\end{lem}

\begin{proof}
We apply the arguments of \refsec{disc_case} to the Euclidean circle packing of the open disc $\widetilde{K}$. The parameters $m_\bullet$ for $\widetilde{\Phi}$ consist of infinitely many copies of the $m_\bullet$ for $K$, and satisfy the circle packing equations for $\widetilde{K}$, which consist of infinitely many copies of the vertex, edge, and triangle equations for $K$. So these equations are satisfied for $K$, and it only remains to check the holonomy equations. 

We demonstrate the holonomy equation \refeqn{holonomy_eqns_1} for $\lambda$; $\mu$ is similar. We adopt the notation of \refsec{circle_packing_eqns}: $\lambda$ consists of $n_\lambda$ normal arcs $\lambda_j$ over $j \in \Z/n_\lambda \Z$, oriented in order of increasing $j$. To avoid awkward notation, drop unnecessary $\lambda$s from our notation, for instance writing $n$ instead of $n_\lambda$. 

The curve $\lambda$ lifts to an oriented normal curve $\widetilde{\lambda}$ in the universal $\widetilde{K}$, which is a bi-infinite sequence of normal arcs $\widetilde{\lambda_j}$ over $j \in \Z$. We number the arcs so that each $\widetilde{\lambda_j}$ in $\widetilde{K}$ covers $\lambda_j$ in $K$. In the fundamental domain $K_0$, each normal arc $\lambda_j$ over $j \in \Z/n \Z$ appears precisely once. (If we regard $K_0$ as $K$ cut along $\mathfrak{l}$ and $\mathfrak{m}$ this is clear; if we regard $K_0$ as a subset of $\widetilde{K}$, these are obtained by applying deck transformations to the $\widetilde{\lambda_j}$.)

For each $j \in \Z/n \Z$, $\lambda_j$ lies in a triangle $\Delta_j$ of $K$ and cuts off a corner $c_j$, anticlockwise or clockwise, and we set $\delta_j = 1$ or $-1$ accordingly. Let $\lambda_j$ begin at the point $p_{j-1}$ on the edge $e_{j-1}$ and end at the point $p_j$ on the edge $e_j$ of $K$. Lifting to $\widetilde{K}$, for each $j \in \Z$, $\widetilde{\lambda_j}$ lies in a triangle $\widetilde{\Delta_j}$, proceeding from the point $\widetilde{p_{j-1}}$ on the edge $\widetilde{e_{j-1}}$, to the point $\widetilde{p_j}$ on the edge $\widetilde{e_j}$ of $\widetilde{K}$, cutting off the corner $\widetilde{c_j}$. Here $\widetilde{\Delta_j}, \widetilde{p_j}$, $\widetilde{e_j}$, $\widetilde{c_j}$ cover $\Delta_j$, $p_j$, $e_j$, $c_j$ respectively.

A ``fundamental domain" of $\widetilde{\lambda}$ (which in general does not lie in the fundamental domain $K_0$ of $\widetilde{K}$) consists of $\widetilde{\lambda_{1}}, \ldots, \widetilde{\lambda_{n}}$, beginning at the point $\widetilde{p_0}$ on $\widetilde{e_0}$, and ending at the point $\widetilde{p_{n}}$ on $\widetilde{e_{n}}$. The initial and final edges $\widetilde{e_0}$ and $\widetilde{e_{n}}$ are related by the deck transformation of $\widetilde{K}$ corresponding to the class of $\lambda$ in $\pi_1 (K)$ (equivalently, $H_1 (K)$ as $K \cong T^2$). (Indeed, this deck transformation sends each $\widetilde{\lambda_j}, \widetilde{\Delta_j}, \widetilde{p_j}, \widetilde{e_j}, \widetilde{c_j}$ to $\widetilde{\lambda_{j+n}}, \widetilde{\Delta_{j+n}}, \widetilde{p_{j+n}}, \widetilde{e_{j+n}}, \widetilde{c_{j+n}}$ respectively.) 

\begin{figure}
\begin{center}
\begin{tikzpicture}[scale=2.2]
 \draw (-2,0) .. controls (-0.5,1) and (0.5,-1) .. (2,0) node[midway,rotate=-22] {\small $>$};
 \draw[dashed, ->] (-2,-0.7) -- (2,-0.7) node[midway,below] {\small $\rho_{\lambda}$};
 \filldraw (-2,0) circle (1pt) node[below left] {\small $\widetilde{p_0}$};
 \filldraw (-1.5,0.235) circle (1pt) node[below right,xshift=0.18em,yshift=0.28em] {\small $\widetilde{p_1}$};
 \filldraw (-0.7,0.24) circle (1pt) node[below right] {\small $\widetilde{p_2}$};
 \filldraw (2,0) circle (1pt) node[below right, xshift=0.5em,yshift=0.2em] {\small $\widetilde{p_n}$};
 \draw[->] (-2,0) -- (-2.2,0.35) node[above left] {\small $\widetilde{e_0}$};
 \draw (-2,0) -- (-1.8,-0.35);
 
 \draw[->] (-1.5,0.235) -- (-1.64,0.235+0.345) node[left] {\small $\widetilde{e_1}$};
 \draw (-1.5,0.235) -- (-1.36,0.235-0.345); 

 \draw[->] (-0.7,0.24) -- (-0.59,0.24+0.36) node[left] {\small $\widetilde{e_2}$};
 \draw (-0.7,0.24) -- (-0.81,0.24-0.36);
 
 \draw[->] (2,0) -- (1.8,0.35) node[above right] {\small $\widetilde{e_n}$};
 \draw (2,0) -- (2.2,-0.35);
 \node at (-1.7,0) {\small $\widetilde{\lambda_1}$};
 \node at (-1.02,0.14) {\small $\widetilde{\lambda_2}$};
 \node at (0,0.3) {\small $\widetilde{\lambda}$};
 \node at (0.75,0) {$\ldots$};
 \node at (1.7,-0.35) {\small $\widetilde{\lambda_n}$};
\end{tikzpicture}
\end{center}
 \caption{ Normal curve $\lambda$ with points $\widetilde{p_j}$ of intersection with edges $\widetilde{e_j}$ of $\widetilde{K}$.}
 \label{Fig:normal_curve_points}
\end{figure}

Under the realisation map $\widetilde{\Phi}$, $\widetilde{e_0}$ and $\widetilde{e_n}$ are related by the holonomy $\rho_\lambda$ of $\lambda$, which is a Euclidean translation as $T^2$ is a Euclidean torus. This translation $\rho_\lambda$ more generally sends each $\widetilde{\Phi}(\widetilde{e_j})$ to $\widetilde{\Phi}(\widetilde{e_{j+n}})$. We may orient each edge $\widetilde{e_j}$ so that at each intersection point $\widetilde{p_j}$, the directions of $\widetilde{\lambda}$ and $\widetilde{e_j}$ form an oriented basis of $\R^2$, as in \reffig{normal_curve_points}. Thus each oriented Euclidean line segment $\widetilde{\Phi}(\widetilde{e_j})$ points in the same direction as $\widetilde{\Phi}(\widetilde{e_{j+n}})$.

At each corner $\widetilde{c_j}$ cut off by $\widetilde{\lambda}$, the realisation $\widetilde{\Phi}$ produces an angle $\theta_{\widetilde{\Delta_j},\widetilde{c_j}}$ and parameter $m_{\widetilde{\Delta_j},\widetilde{c_j}}$, which for convenience we simply denote $\theta_j$ and $m_j$. Because the translation $\rho_\lambda$ sends each corner $\widetilde{c_j}$ to $\widetilde{c_{j+n}}$ we have $\theta_j = \theta_{j+n}$ and $m_j = m_{j+n}$ for all $j \in \Z$. In particular, $\theta_j$ and $m_j$ are well defined for each $j \in \Z/n\Z$; $\theta_j$ is the angle, and $m_j$ the corresponding parameter, of the original realisation $\Phi \colon K \To T^2$ at the corner $c_j$ of $\Delta_j$.

We observe that the direction of $\widetilde{e_j}$ is obtained from the direction of $\widetilde{e_{j-1}}$ by a rotation of $\delta_j \theta_j$, i.e. by $\pm \theta_j$ accordingly as $\widetilde{\lambda_j}$ proceeds around $\widetilde{c_j}$ anticlockwise or clockwise.

As $\widetilde{e_0}, \widetilde{e_{n}}$ have the same direction then $\sum_{j=1}^{n} \delta_j \theta_j$ is an integer multiple of $2\pi$. By \reflem{exp_Agol} then $\prod_{j=1}^{n} \left( m_j + i \right)^{\delta_j}$ has argument an integer multiple of $\pi$, as does its positive multiple $\prod_{j=1}^{n} \left( m_j + \delta_j i \right)$. Expanding out the imaginary part of this product to zero, the holonomy equation for $\lambda$ is satisfied.

If there is no branching, then by the branched disc case applied to $\widetilde{K}$, the unbranched vertex equations are satisfied. The realisation map $\widetilde{\Phi}$ is then a homeomorphism $\widetilde{K} \To \R^2$. Proceeding along $\widetilde{\lambda}$ from $\widetilde{p_0}$ through to $\widetilde{p_{n}}$, the overall rotation is $0$, rather than a general multiple of $2\pi$, so $\sum_{j=1}^{n} \delta_j \theta_j = 0$, hence $\sum_{j=1}^{n} \arg \left(m_j + i \right)^{\delta_j}$ is zero, as is $\sum_{j=1}^{n} \arg \left( m_j + \delta_j i \right)$. So the unbranched holonomy equation \refeqn{unbranched_holonomy_eqns_1} is satisfied.
\end{proof}

We have now essentially proved half of \refthm{circle_packing_eqs_general} (including the more precise \refthm{circle_packing_eqs_spherical}) and \refthm{circle_packing_eqs_general_unbranched}, showing that conformal classes of appropriate circle packings yield positive real solutions to circle packing equations.
It thus remains to prove the converse of the two main theorems, constructing unique conformal classes of circle packings from solutions to the equations.

\section{From solutions to circle packings}
\label{Sec:solutions_to_packings}

In this section we show that solutions to the circle packing equations describe circle packings. We consider the various possibilities for the complex $K$ in turn: discs in \refsec{closed_disc} and \refsec{discs_in_general}, spheres in \refsec{spheres}, and tori in \refsec{tori}. Note that in the disc case, $K$ is a finite complex precisely when it triangulates a closed disc, and infinite precisely when it triangulates an open or semi-open disc. We deal with the finite case in \refsec{closed_disc} and infinite case in \refsec{discs_in_general}.

\subsection{Finite discs}
\label{Sec:closed_disc}

 Suppose $K$ is a finite simplicial complex triangulating an oriented closed disc. Let $\{m_{\Delta,v}\}_{(\Delta,v) \in C(K)}$ be a solution of the circle packing equations for $K$, with all $m_{\Delta,v} > 0$. We will show that there is a Euclidean circle packing of $K$ with realisation $\Phi \colon K \to \R^2$, curvatures $\{\kappa_v\}_{v \in V(K)}$ and angles $\{\theta_{\Delta,v}\}_{(\Delta,v) \in C(K)}$, such that each $m_{\Delta,v}$ satisfies
\begin{equation}
\label{Eqn:desired_relations}
m_{\Delta,v} = \cot \left( \frac{\theta_{\Delta,v}}{2} \right)
= \frac{\sqrt{\kappa_u \kappa_v + \kappa_v \kappa_w + \kappa_w \kappa_u}}{\kappa_v},
\end{equation}
where $u,v,w$ are the vertices of $\Delta$. Moreover, we will show that this circle packing is unique up to Euclidean similarities. We will further show that if the $m_{\Delta,v}$ satisfy the unbranched circle packing equations, then this circle packing is unbranched.

Given the solution $\{m_{\Delta,v}\}_{(\Delta,v) \in C(K)}$, we can define angles $\{\theta_{\Delta,v}\}_{(\Delta,v) \in C(K)}$ by \refeqn{desired_relations}, with all $\theta_{\Delta,v} \in (0, \pi)$. For now they are just numbers derived from solutions to circle packing equations, but we will show they are realised as the angles in $\Phi$.

Let $N$ be the number of triangles in $K$. Our proof is by induction on $N$, using an idea similar to \cite[section 6.3]{StephensonKenneth2005Itcp}.
\begin{lem}
\label{Lem:build_triangulation}
If $N \geq 2$, then at least one of the following occurs.
\begin{enumerate}
\item 
There is an interior edge $e$ of $K$, both of whose endpoints are on the boundary of $K$. 
\item 
There is a triangle $\Delta$ of $K$ which contains a boundary edge of $K$ and an interior vertex of $K$.
\end{enumerate}
\end{lem}

\begin{proof}
Suppose that (1) never arises. Then every interior edge of $K$ has at least one endpoint which is an interior vertex. Take an arbitrary boundary edge of $K$, with endpoints $v,w$, let $\Delta$ be the adjacent triangle, and let the vertices of $\Delta$ be $u,v,w$. As $N \geq 2$, at least one of the edges $uv$, $uw$ is an interior edge of $K$. If $u \in \partial K$ then $\Delta$ has an edge which is interior to $K$ and has both endpoints on the boundary of $K$, contradicting our assumption. Thus $u$ is an interior vertex, and $\Delta$ satisfies (2).
\end{proof}

In case (1) of \reflem{build_triangulation}, the edge $e$ splits $K$ into two smaller complexes $K_1, K_2$, which are both simplicial complexes triangulating oriented discs, and $K$ is obtained by gluing $K_1$ and $K_2$ along $e$. In case (2), removing $\Delta$ yields a smaller complex $K'$, and $K$ is obtained from $K'$ by gluing the triangle $\Delta$ along two consecutive boundary edges of $K'$. See \reffig{build_triangulation}.

\begin{figure}
\begin{center}
\begin{tabularx}{0.8\linewidth} { 
 >{\centering\arraybackslash}X 
 >{\centering\arraybackslash}X 
 >{\centering\arraybackslash}X }
\begin{tikzpicture}
 \node[draw, minimum size=3.5cm,regular polygon,regular polygon sides=14] (a) {};
 \foreach \x in {1,2,...,14}
 \fill (a.corner \x) circle[radius=1pt];
 \draw (a.corner 12)--(a.corner 5);
 \node[above] at (0,0) {$e$};
 \node[above] at (0,0.8) {$K_1$};
 \node[below] at (0,-0.8) {$K_2$};
\end{tikzpicture} &
\begin{tikzpicture}
 \node[rotate=10, minimum size=3.5cm,regular polygon,regular polygon sides=9] (a) {};
 \foreach \x in {1,2,...,9}
 \fill (a.corner \x) circle[radius=1pt];
 \foreach \y [count=\z from 2] in {1,2,...,9}
 \draw (a.corner \y)--(a.corner \z);
 \draw (0.5,0.1)--(a.corner 8);
 \draw (0.5,0.1)--(a.corner 7);
 \node[right] at (0.9,0.1) {$\Delta$};
 \node at (0,0) {$K'$};
\end{tikzpicture}
\end{tabularx}
\end{center}
 \caption{Cases (1) and (2) of \reflem{build_triangulation}.}
 \label{Fig:build_triangulation}
\end{figure}

\begin{proof}[Proof of \refthm{circle_packing_eqs_general} and \refprop{m_meaning} when $K$ is a finite disc]
If $N=1$ then $K$ consists of a single triangle, so there are precisely $3$ variables $m_1, m_2, m_3$ and a single circle packing equation, that for a single triangle. Let $\{m_1, m_2, m_3\}$ be a solution with all $m_j > 0$, and let $\theta_1, \theta_2, \theta_3 \in (0, \pi)$ be defined by \refeqn{desired_relations}, i.e. $m_j = \cot(\theta_j/2)$. As $m_1, m_2, m_3$ satisfy the triangle equation, by \reflem{triangle_equation_angle_sum} then $\theta_1 + \theta_2 + \theta_3 = \pi$. So $\theta_1, \theta_2, \theta_3$ are the angles of a Euclidean triangle. Let $\Phi \colon K \To \R^2$ realise this triangle, and draw its Soddy circles, with curvatures $\kappa_1, \kappa_2, \kappa_3$. By \reflem{edge_eqn_fact}, our solution $m_\bullet$ coincides with the parameters $m_\bullet$ as defined in \refsec{Euclidean_triangles} by $m_j = \sqrt{L_\Delta}/{\kappa_j}$ where $L_\Delta = \kappa_1 \kappa_2 + \kappa_2 \kappa_3 + \kappa_3 \kappa_1$. Thus we have a circle packing as claimed, well defined up to similarity, and \refeqn{desired_relations} holds. Indeed, any $\Phi$ realising this solution must have angles $\theta_1, \theta_2, \theta_3$ and so the realisation is unique up to similarity.

Now suppose $N \geq 2$, so \reflem{build_triangulation} applies and we consider the two cases in turn. Let $\{m_{\Delta,v}\}_{(\Delta,v) \in C(K)}$ be a solution to the circle packing equations for $K$ with all $m_{\Delta,v}>0$, and let $\{\theta_{\Delta,v}\}_{(\Delta,v) \in C(K)}$ be defined by \refeqn{desired_relations}, i.e. $m_{\Delta,v} = \cot(\theta_{\Delta,v}/2)$, with all $\theta_{\Delta,v} \in (0, \pi)$.

Suppose $K$ is obtained from two smaller complexes $K_1, K_2$, by gluing along an edge $e$. Then the circle packing equations for $K$ consist of those for $K_1$ and $K_2$, with no new variables, and one new edge equation. Letting the triangles adjacent to $e$ in $K_1, K_2$ be $\Delta_1, \Delta_2$, and the endpoints of $e$ be $v_1, v_2$, the situation is as in \reffig{circle_packing_eqn_vars} (centre) and the new edge equation is $m_{\Delta_1,v_1} m_{\Delta_2,v_2} = m_{\Delta_1,v_2} m_{\Delta_2,v_1}$. So the solution $\{m_{\Delta,v}\}_{(\Delta, v) \in C(K)}$ of the circle packing equations for $K$ contains solutions $\{m_{\Delta,v}\}_{(\Delta,v) \in C(K_1)}$ and $\{m_{\Delta,v}\}_{(\Delta,v) \in C(K_2)}$ of the circle packing equations for $K_1$ and $K_2$. By induction, these yield realisations $\Phi_1 \colon K_1 \To \R^2$ and $\Phi_2 \colon K_2 \To \R^2$, unique up to similarity, with circle radii $\{r_v^1\}_{v \in V(K_1)}$ and $\{r_v^2\}_{v \in V(K_2)}$ and curvatures $\{\kappa_v^1\}_{v \in V(K_1)}$ and $\{\kappa_v^2\}_{v \in V(K_2)}$ respectively, such that the relations of \refeqn{desired_relations} hold and the $\theta_{\Delta,v}$ are realised as the angles arising in $\Phi_1$ and $\Phi_2$.

In $\Phi_1$, the circles corresponding to the endpoints $v_1, v_2$ of $e$ have radii $r_{v_1}^1, r_{v_2}^1$ and curvatures $\kappa_{v_1}^1, \kappa_{v_2}^1$ respectively, so the length of $e$ is $r_{v_1}^1 + r_{v_2}^1$. Note also that by \reflem{edge_eqn_fact}
\[
\frac{m_{\Delta_1,v_1}}{m_{\Delta_1,v_2}} 
= \frac{\kappa_{v_2}^1}{\kappa_{v_1}^1}
= \frac{r_{v_1}^1}{r_{v_2}^1}. 
\]
Similarly in $\Phi_2$, the circles corresponding to $v_1, v_2$ have radii $r_{v_1}^2, r_{v_2}^2$ and curvatures $\kappa_{v_1}^2, \kappa_{v_2}^2$ respectively, so the length of $e$ is $r_{v_1}^2 + r_{v_2}^2$, and
\[
\frac{m_{\Delta_2,v_1}}{m_{\Delta_2,v_2}}
= \frac{\kappa_{v_2}^2}{\kappa_{v_1}^2}
= \frac{r_{v_1}^2}{r_{v_2}^2}.
\]
By the new edge equation $m_{\Delta_1,v_1} m_{\Delta_2,v_2} = m_{\Delta_1,v_2} m_{\Delta_2,v_1}$ we thus have
\[
\frac{m_{\Delta_1,v_1}}{m_{\Delta_1,v_2}} = \frac{m_{\Delta_2,v_1}}{m_{\Delta_2,v_2}}, 
\quad \text{i.e.} \quad
\frac{\kappa_{v_1}^1}{\kappa_{v_2}^1} = \frac{\kappa_{v_1}^2}{\kappa_{v_2}^2}
\quad \text{and} \quad
\frac{r_{v_1}^1}{r_{v_2}^1} = \frac{r_{v_1}^2}{r_{v_2}^2}.
\]

Thus the circles about vertices $v_1, v_2$ in $\Phi_1$ and in $\Phi_2$ are in the same ratio. Hence there is a unique dilation of $\Phi_1$ after which their curvatures are equal, $\kappa_{v_1}^1 = \kappa_{v_1}^2$ and $\kappa_{v_2}^1 = \kappa_{v_2}^2$, as are their radii, $r_{v_1}^1 = r_{v_1}^2$ and $r_{v_2}^1 = r_{v_2}^2$. Then the side lengths of $e$ in $\Phi_1$ and $\Phi_2$ are equal, $r_{v_1}^1 + r_{v_2}^1 = r_{v_1}^2 + r_{v_2}^2$. Adjusting $\Phi_1$ by a further orientation-preserving isometry, we may then glue $\Phi_1$ and $\Phi_2$ together along $e$ to obtain an realisation $\Phi$ of $K$. In this realisation, the circles at $v_1, v_2$ from $\Phi_1, \Phi_2$ are identified. Since we only adjusted the radii and curvatures of $\Phi_1$ by an overall dilation, the variables $\theta_{\Delta,v}$ are still realised as angles. Since the ratios of curvatures in any individual triangle remain unchanged, the relations \refeqn{desired_relations} continue to hold. Since $\Phi_1$ and $\Phi_2$ are unique up to similarity, their gluing $\Phi$ along $e$ is also unique up to similarity.

Finally, suppose that $K$ is obtained from a smaller complex $K'$ by gluing a triangle $T$ along two consecutive boundary edges $e_1, e_2$ of $K'$, with endpoints $u_1, w$ and $u_2, w$ respectively. The circle packing equations for $K$ then consist of those for $K'$ and those for the single triangle $T$, with no new variables, one new vertex equation at $w$, and two new edge equations for $e_1$ and $e_2$.
The solution $\{m_{\Delta,v}\}_{(\Delta,v) \in C(K)}$ yields solutions $\{m_{\Delta,v}\}_{(\Delta,v) \in C(K')}$ and $\{m_{\Delta,v}\}_{(\Delta,v) \in C(T)}$ of the circle packings for $K'$ and $T$. By induction, we obtain realisations $\Phi' \colon K' \To \R^2$ and $\Phi^T \colon T \To \R^2$, unique up to similarity, with circle radii $\{r'_v\}_{v \in V(K')}$ and $\{r^T_v\}_{v \in V(T)}$ and curvatures $\{\kappa'_v\}_{v \in V(K')}$ and $\{\kappa^T_v\}_{v \in V(T)}$, such that the relations of \refeqn{desired_relations} hold and the $\theta_{\Delta,v}$ are realised as angles.

At $w$, we have at least one corner of $K'$, and the corner $(T, w)$ of $T$. As the relations \refeqn{desired_relations} hold, by \reflem{vertex_equation_angle_sum} the vertex equation at $w$ implies that $\theta_{T,w}$ sums with the $\theta_{\Delta,w}$ in corners of $K'$ at $w$ to give a positive integer multiple of $2\pi$. As the $\theta_{\Delta,v}$ are realised as angles in $\Phi'$ and $\Phi^T$, then $\theta_{T,w}$ is just the right angle for $\Phi^T$ to slot into $\Phi'$ at $w$.

In $\Phi'$, the edges $e_1, e_2$ have lengths $r'_{u_1} + r'_v$ and $r'_{u_2} + r'_v$ respectively. As in the previous case, the edge equation for $e_1$ implies that
\[
\frac{\kappa'_{u_1}}{\kappa'_w} = \frac{\kappa^T_{u_1}}{\kappa^T_w}
\quad \text{and} \quad
\frac{r'_{u_1}}{r'_w} = \frac{r^T_{u_1}}{r^T_w}.
\]
Similarly, the edge equation for $e_2$ implies that
\[
\frac{\kappa'_{u_2}}{\kappa'_w} = \frac{\kappa^T_{u_2}}{\kappa^T_w}
\quad \text{and} \quad
\frac{r'_{u_2}}{r'_v} = \frac{r^T_{u_2}}{r^T_v}.
\]
Thus we have equalities of ratios
\[
\left[ \kappa'_{u_1} : \kappa'_w : \kappa'_{u_2} \right]
= \left[ \kappa^T_{u_1} : \kappa^T_w : \kappa^T_{u_2} \right]
\quad \text{and} \quad
\left[ r'_{u_1} : r'_w : r'_{u_2} \right]
= \left[ r^T_{u_1} : r^T_w : r^T_{u_2} \right],
\]
so the circles at $u_1, w, u_2$ in $\Phi'$ and $\Phi^T$ are in the same ratio. As in the previous case, there is then a unique dilation of $\Phi^T$ after which $(\kappa^T_{u_1}, \kappa^T_w, \kappa^T_{u_2}) = (\kappa'_{u_1}, \kappa'_w, \kappa'_{v_2})$ and $(r^T_{u_1}, r^T_w, r^T_{u_2}) = (r'_{u_1}, r'_w, r'_{u_2})$. The side lengths of $e_1$ and $e_2$ then agree in $\Phi'$ and $\Phi^T$, so we may then glue $\Phi'$ and $\Phi^T$ together to obtain a realisation $\Phi$ of $K$, in which the circles at $u_1, u_2$ and $w$ are identified. As in the previous case, the variables $\theta_{\Delta,v}$ are still realised as angles, the relations \refeqn{desired_relations} continue to hold, and $\Phi$ is unique up to similarity.
\end{proof}

\begin{proof}[Proof of \refthm{circle_packing_eqs_general_unbranched} when $K$ is a finite disc]
The proof is identical to above, except that in the final case, when gluing around the vertex $w$, the unbranched vertex equation $\sum \arg (m_\bullet + i) = \pi$ is equivalent to $\sum \theta_\bullet = 2\pi$, since $m_\bullet = \cot(\theta_\bullet/2)$. Thus the angle sum around $w$ is exactly $2\pi$, and $\Phi'$ and $\Phi^T$ slot together without branching.
\end{proof}

Note that the above argument still works if we allow triangles with $m_{\Delta,v}$ negative, in which case the angles $\theta_{\Delta,v}$ are also negative, and we regard those triangles as negatively oriented. The realisation $\Phi$ then has fold singularities along edges between triangles of opposite orientation.

We also observe that the arguments above rely crucially on $K$ being simply connected, so that circle packings can be built over larger and larger discs. If $K$ were not simply connected, we would have to verify that $\Phi$ glues triangles appropriately around non-contractible curves, and this will not be true in general. We will see in \refsec{tori} that holonomy equations guarantee this in the torus case.

\subsection{ Infinite discs }
\label{Sec:discs_in_general}

Suppose more generally now that $K$ is an infinite simplicial complex triangulating an open or semi-open disc. As $K$ is simplicial, we denote corners by $(\Delta,v)$, where $\Delta \in F(K)$ and $v$ is a vertex of $\Delta$. As $K$ is a surface, its set of triangles is countable, and every interior vertex has finite degree. We may take a sequence of finite simplicial sub-complexes
\[
K_1 \subset K_2 \subset K_3 \subset \cdots
\]
such that for each positive integer $n$, $K_n$ triangulates a closed disc, and $\bigcup_{n=1}^\infty K_n = K$. Note that whenever we have positive integers $n < n'$, not only do we have inclusions of complexes $K_n \subset K_{n'} \subset K$ but also inclusions of sets $V(K_n) \subseteq V(K_{n'}) \subset V(K)$, $F(K_n) \subseteq F(K_{n'}) \subset F(K)$ and $C(K_n) \subseteq C(K_{n'}) \subset C(K)$. Moreover, every interior vertex of $K$ is also an interior vertex of $K_n$ for sufficiently large $n$.

Let $\{m_{\Delta,v}\}_{(\Delta,v) \in C(K)}$ be a solution to the circle packing equations for $K$, with all $m_{\Delta,v} > 0$. Note this set of numbers is infinite, as is the set of circle packing equations, but each equation is a polynomial equation, involving only finitely many variables. We will show there is a Euclidean circle packing of $K$ with a realisation map $\Phi \colon K \To \R^2$, unique up to Euclidean similarity, with curvatures $\{\kappa_v\}_{v \in V(K)}$, and angles $\{\theta_{\Delta,v}\}_{(\Delta,v) \in C(K)}$ satisfying \refeqn{desired_relations}.

\begin{proof}[Proof of \refthm{circle_packing_eqs_general} and \refprop{m_meaning} when $K$ is an infinite disc]
For each positive integer $n$, the circle packing equations of $K_n$ are a finite subset of the circle packing equations of $K$ in a finite subset of the variables $m_{\Delta,v}$. Therefore the subset $\{m_{\Delta,v}\}_{(\Delta,v) \in C(K_n)}$ of $\{m_{\Delta,v}\}_{(\Delta,v) \in C(K)}$ satisfy the circle packing equations of $K_n$. By the arguments of \refsec{closed_disc} there exists a circle packing for $K_n$, with a realisation map $\Phi_n \colon K_n \To \R^2$, unique up to Euclidean similarity, such that the curvatures and angles of $\Phi_n$ (and the parameters $\{m_{\Delta,v}\}_{(\Delta,v) \in C(K_n)}$) satisfy \refeqn{desired_relations}.

When $n < n'$, the restriction of the circle packing of $K_{n'}$ to $K_n$ has the same parameters $m_{\Delta,v}$ as the circle packing for $K_n$, so again by \refsec{closed_disc} these two circle packings agree, up to a Euclidean similarity. Thus $\Phi_{n'}|_{K_n}$ agrees with $\Phi_n$, up to a Euclidean similarity.

Starting from the circle packing of $K_1$ and successively adjusting the circle packings of $K_2, K_3, \ldots$, or equivalently each $\Phi_2, \Phi_3, \ldots$, by a Euclidean similarity, we can arrange that for all $n < n'$, the circle packing of $K_{n}$ is a subset of the circle packing of $K_{n'}$. Equivalently, we have $\Phi_{n'} |_{K_n} = \Phi_n$, and can write 
\[
\Phi_1 \subset \Phi_2 \subset \Phi_3 \subset \cdots.
\]
There is then a well-defined map $\Phi \colon K \To \R^2$ given by $\Phi = \bigcup_{n=1}^\infty \Phi_n$, the realisation of a circle packing for $K$ obtained as the union of the circle packings for the $K_n$. Let the curvatures of the circles in this circle packing be $\{\kappa_v\}_{v \in V(K)}$, and let the angles in corners given by $\Phi$ be $\{\theta_{\Delta,v}\}_{(\Delta,v) \in C(K)}$.

Consider a vertex $v$ and corner $(\Delta, v)$ of $K$. These are also a vertex and corner of some $K_n$. Since $\Phi_n = \Phi|_{K_n}$, the circle packings of $K_n$ and $K$ have the same angle $\theta_{\Delta,v}$ at this corner, and the same circle curvature $\kappa_v$. Again by \refsec{closed_disc}, the parameters $m_{\Delta,v}$, angles $\theta_{\Delta,v}$ and curvatures $\kappa_v$ satisfy \refeqn{desired_relations}. Thus $\Phi$ has the desired properties.

To see that $\Phi$ is unique up to conformal equivalence, let $\Phi, \Phi' \colon K \To \R^2$ be two realisations of circle packings of $K$ with the desired properties. For each positive integer $n$, let $\Phi_n = \Phi|_{K_n}$ and $\Phi'_n = \Phi'|_{K_n}$. Then $\Phi_n, \Phi'_n$ are realisations of circle packings of $K_n$ with the same parameters $\{m_{\Delta,v}\}_{(\Delta,v) \in C(K_n)}$. 
By uniqueness in the finite case, there exists a Euclidean similarity $\Psi_n$ such that $\Phi'_n = \Psi_n \circ \Phi_n$. But for two circle packings each involving at least three circles, if there is a similarity relating them, that similarity is unique. Since the $\Phi_n$ and $\Phi'_n$ form ascending series of circle packings, the similarities $\Psi_n$ are all equal, independent of $n$. Denoting this similarity by $\Psi$, we have $\Phi'_n = \Psi \circ \Phi_n$ for all $n$, hence $\Phi' = \Psi \circ \Phi$, so the two circle packings of $K$ are conformally equivalent.
\end{proof}

\begin{proof}[Proof of \refthm{circle_packing_eqs_general_unbranched} when $K$ is an infinite disc]
When $\{m_{\Delta,v}\}_{(\Delta,v) \in C(K)}$ also satisfy the unbranched vertex equations, by the closed disc case, each $\Phi_n$ is unbranched at each interior vertex of $K_n$. For each interior vertex $v$ of $K$, $v$ is also an interior vertex of some $K_n$, so $\Phi_n$ is unbranched at $v$, and hence $\Phi = \bigcup_{n=1}^\infty \Phi_n$ is also unbranched at $v$.
\end{proof}

\subsection{Spheres}
\label{Sec:spheres}

Suppose now $K$ is a simplicial complex triangulating a sphere. As in \refsec{circle_packing_eqns} and \refsec{spherical_case_eqns} let $\Delta_0$ be a triangle of $K$, and let $K_0$ be the finite simplicial complex triangulating an oriented closed disc obtained by removing the interior of $\Delta_0$.

Let $e_0$ be an edge of $\Delta_0$. We then have three sets of circle packing equations: the reduced circle packing equations for $(K, \Delta_0, e_0)$, which we denote $\mathcal{E}_{red}$ (\refdef{reduced_eqns_S2}); the full set of circle packing equations for $(K, \Delta_0)$, which we denote $\mathcal{E}$ (\refdef{circle_packing_eqns_sphere}); and the circle packing equations for $K_0$, which we denote $\mathcal{E}_0$ (\refdef{circle_packing_eqns_disc}). We regard the ``variables" in the corners of $\Delta_0$ as constants $-\sqrt{3}$, that is, not as variables at all. The sphere-closing equations \refeqn{sphere-closing_eqn_list} are then satisfied automatically, as is the triangle equation \refeqn{triangle_eqn_list} in $\Delta_0$. The sets of equations $\mathcal{E}_0, \mathcal{E}_{red}, \mathcal{E}$ then all have the same variables, corresponding to the corners of $K_0$, and we do not regard them as containing sphere-closing equations, or the triangle equation for $\Delta_0$. Thus 
\[
\mathcal{E}_0 \subset \mathcal{E}_{red} \subset \mathcal{E}.
\]
As in \refsec{spherical_case_eqns}, denote the vertices of $\partial \Delta_0$ in order as $v_0, v_1, v_2$. Let $e_j$ be the edge of $\Delta_0$ opposite $v_j$. We may choose these labels so that $e_0$ is the distinguished edge in $\mathcal{E}_{red}$. Then $\mathcal{E}_{red}$ is obtained from $\mathcal{E}_0$ by adding the edge equations for $e_1$ and $e_2$; and $\mathcal{E}$ is obtained by adding the final edge equation at $e_0$, and $3$ vertex equations around $\Delta_0$.

\begin{proof}[Proof of \refprop{circle_packing_reduction}]
Clearly any solution of $\mathcal{E}$ also satisfies the subset of equations $\mathcal{E}_{red}$. So we take a solution $\{m_{\Delta,v}\}_{(\Delta,v) \in C(K_0)}$ of $\mathcal{E}_{red}$ and show it satisfies the equations of $\mathcal{E}$.

\begin{figure}[!h]
 \centering
 \begin{tikzpicture}[scale=2]
 \coordinate (A) at (0,1.1);
 \coordinate (B) at (-1.3,-0.9); 
 \coordinate (C) at (1.6,-0.9);
 \coordinate (O) at (0,0);
 \draw[thick] (A) node[above] {$v_0$};
 \draw[thick] (B) node[left] {$v_1$};
 \draw[thick] (C) node[right] {$v_2$};
				\draw[thick] (A) -- node[midway, above left] {$e_2$} (B);
 \draw[thick] (B) -- node[midway, below] {$e_0$} (C);
 \draw[thick] (C) -- node[midway, above right] {$e_1$} (A);
				\draw[thick] (A) -- (-0.04,0.6);
				\draw[thick] (A) -- (0.04,0.6);
 \pic [draw, ->, "$m_{2,0}$", draw=none, angle eccentricity=3] {angle = B--A--O};
				\pic [draw, ->, "$m_{1,0}$", draw=none, angle eccentricity=3] {angle = O--A--C};
				\draw[thick] (B) -- (-0.75,-0.55);
				\draw[thick] (B) -- (-0.7,-0.6);
				\draw (-0.7, -0.4) node {$m_{2,1}$};
				\draw (-0.5,-0.8) node {$m_{0,1}$};
				\draw[thick] (C) -- (1.05,-0.55);
				\draw[thick] (C) -- (1,-0.6);
				\draw (0.96,-0.4) node {$m_{1,2}$};
				\draw (0.8,-0.8) node {$m_{0,2}$};
 \end{tikzpicture}
 \caption{Labels near $\Delta_0$ in $K_0$. Note the vertices $v_j$ may have arbitrarily high degree, and if $v_j$ has degree $2$ then two labels $m_{\bullet,j}$ coincide.}
 \label{Fig:labels_near_D0}
\end{figure}

Since $\mathcal{E}_0 \subset \mathcal{E}_{red}$, these $m_\bullet$ form a solution of the circle packing equations for $K_0$, hence by \refsec{closed_disc} 
yield a Euclidean circle packing of $K_0$, with realisation $\Phi \colon K_0 \To \R^2$ and circle curvatures $\{\kappa_v\}_{v \in V(K_0)}$. Denote by $m_{i,j}$ the number in the triangle of $K_0$ with edge $e_i$ at vertex $v_j$, as in \reffig{labels_near_D0}. 
The three edge equations along $\Delta_0$ are then
\begin{equation}
\label{Eqn:edge_eqns_around_D0}
-\sqrt{3} \; m_{0,1} = - \sqrt{3} \; m_{0,2}, \quad
-\sqrt{3} \; m_{1,2} = - \sqrt{3} \; m_{1,0}, \quad
-\sqrt{3} \; m_{2,0} = - \sqrt{3} \; m_{2,1},
\end{equation}
all of which are in $\mathcal{E}$, but only the latter two are in $\mathcal{E}_{red}$. By \refprop{m_meaning} and \reflem{edge_eqn_fact} in the closed disc case, these equations give ratios of the curvatures $\kappa_{v_j}$, and are respectively equivalent to
\begin{equation}
\label{Eqn:curvatures_equal_around_D0}
\frac{\kappa_{v_2}}{\kappa_{v_1}} = 1, \quad
\frac{\kappa_{v_0}}{\kappa_{v_2}} = 1, \quad
\frac{\kappa_{v_1}}{\kappa_{v_0}} = 1. \quad
\end{equation}
Any two of these equations imply the third, hence the $m_\bullet$ satisfy all three of them. Moreover, $\kappa_{v_0} = \kappa_{v_1} = \kappa_{v_2}$, so the three circles at the $v_j$ in the circle packing of $K_0$ realising the $m_\bullet$ are isometric, and their centres form an equilateral triangle.

Thus, for each vertex $v_j$ of $\Delta_0$, the variables $m_{\Delta,v_j}$ at corners of triangles $\Delta$ of $K_0$ at $v_j$ have associated angles $\theta_{\Delta,v_j}$ summing to $\pi/3$. By \refprop{m_meaning} in the closed disc case, each $m_\bullet = \cot (\theta_\bullet/2)$. By \reflem{exp_Agol}, each $m_\bullet + i$ has argument $\theta_\bullet / 2$. Hence 
\[
\arg \prod_{\Delta} \left( m_{\Delta,v_j} + i \right) = \frac{\pi}{6}.
\]
The constant $m_{v_j,\Delta_0} = -\sqrt{3}$ imposed by the sphere-closing equation then yields $m_{v_j,\Delta_0} + i = - \sqrt{3} + i$, which has argument $5\pi/6$. We thus obtain
\[
\Im \left( - \sqrt{3} + i \right) \prod_{\Delta} \left( m_{\Delta,v_j} + i \right) = 0,
\]
which is precisely the vertex equation at $v_j$ in $\mathcal{E}$. Hence the $\{m_{\Delta,v}\}$ satisfy all the equations of $\mathcal{E}$ as desired.
\end{proof}

The choice of constants $-\sqrt{3}$ is not necessary. The above argument works with any choice of negative constants satisfying the triangle equation, but equal constants save some effort.

We can now prove the main theorem in the spherical case. 
In \refsec{spherical_case_eqns}, given a spherical circle packing of $K$ good for $\Delta_0$, by removing $\Delta_0$, applying a M\"{o}bius transformation $M$ to $S^2$, then applying stereographic projection $\mathcal{S}$, we obtained a circle packing of $\overline{K_0}$ 
in $\R^2$ whose realisation has image consisting of an equilateral triangle centred at the origin and its interior. From this we obtained parameters $\{m_{\Delta,v}\}$ over $(\Delta,v) \in C(K)$ satisfying the circle packing equations for $(K,\Delta_0)$. Moreover we showed that these $m_\bullet$ depended only on the conformal class of the circle packing. 
It remains to show that this process can be reversed, constructing circle packings given positive real solutions of the circle packing equations for $(K, \Delta_0)$, to provide a bijection between these solutions and conformal classes of spherical circle packings realising $K$ which are good for $\Delta_0$. 

\begin{proof}[Proof of \refthm{circle_packing_eqs_general}, \refthm{circle_packing_eqs_spherical} and \refthm{circle_packing_eqs_general_unbranched} when $K \cong S^2$]
As $C(\overline{K_0}) \subset C(K)$ and the circle packing equations of $\overline{K_0}$ are a subset of those for $K$ (we identify the corners of $K_0$ and $\overline{K_0}$; the circle packing equations do not depend on the orientation of the complex), the given numbers $\{m_{\Delta,v}\}_{(\Delta,v) \in C(\overline{K_0})}$ solve the circle packing equations for $\overline{K_0}$. Thus by \refsec{closed_disc}, there is a circle packing of $\overline{K_0}$ in the Euclidean plane $\R^2$, with orientation-preserving realisation map $\overline{\Phi} \colon \overline{K_0} \To \R^2$, whose angles and curvatures are related to the $m_{\Delta,v}$ by \refeqn{desired_relations}. 

The boundary of $\overline{K_0}$ is a triangle and we consider its image under $\overline{\Phi}$. We adopt the notation of \reffig{labels_near_D0} and the proof of \refprop{circle_packing_reduction} for edges, vertices and parameters near $\Delta_0$. The $m_\bullet$ satisfy the edge equations \refeqn{edge_eqns_around_D0} around $\Delta_0$, so as in the proof of \refprop{circle_packing_reduction} we have \refeqn{curvatures_equal_around_D0}. Hence the circles at $v_0, v_1, v_2$ are all isometric, $\overline{\Phi}$ maps $\partial \overline{K_0}$ to an equilateral triangle, and the image of $\overline{\Phi}$ consists of this triangle and its interior.

Applying a Euclidean translation, we can assume this triangle has its centroid at the origin in $\R^2$. Applying stereographic projection $\mathcal{S}^{-1} \colon \R^2 \To S^2 \setminus \{N\}$, we obtain a circle packing of $K_0$ (as $\mathcal{S}$ reverses orientation) in $S^2$ with realisation $\Phi_0 \colon K_0 \To S^2$ whose image is a geodesic triangle of $S^2$ and its interior, avoiding the north pole $N$. Denoting its circles by $\{C_v\}_{v \in V(K_0)}$, each $C_v$ bounds a disc containing $\Phi_0 (v)$ disjoint from $N$.

The circles of this packing in fact yield a spherical circle packing of $K$, with realisation $\Phi \colon K \To S^2$ extending $\Phi_0$, mapping $\Delta_0$ to a geodesic triangle containing $N$. Then $N$ lies in the interstice of $\Delta_0$ and this circle packing is good for $\Delta_0$. The method of \refsec{spherical_case_eqns} applied to this circle packing recovers the the original parameters $m_{\Delta,v}$, so we have a bijection.

If the $m_\bullet$ satisfy unbranched circle packing equations, by \refsec{closed_disc} the packing of $\overline{K_0}$ is unbranched. After stereographic projection and inserting $\Delta_0$, it remains unbranched.
\end{proof}

Note that, after constructing a circle packing good for $\Delta_0$, we can also construct circle packings good for other triangles $\Delta'_0$ of $K$, in the same conformal class, simply by applying M\"{o}bius transformations to $S^2$. Thus, the circle packing equations for $(K, \Delta_0)$ and $(K, \Delta'_0)$ are distinct but lead to conformally equivalent spherical circle packings. Our method, based on stereographic projection, privileges the triangle covering the north pole, which maps to the ``exterior" triangle in the plane under stereographic projection. 

As noted in the disc case, if we allow parameters $m_\bullet$ and angles $\theta_\bullet$ in some triangles to be negative, we can regard those triangles as negatively oriented. Under stereographic projection, such triangles can be regarded as ``turned inside out" to cover the north pole.

\subsection{Tori}
\label{Sec:tori}

Let $K$ be a $\Delta$-complex triangulating a torus, with universal cover $\widetilde{K}$ a simplicial complex triangulating $\R^2$. As in \refsec{circle_packing_eqns} and \refsec{torus_eqns_satisfied}, let $K_0$ be a fundamental domain complex obtained by cutting along simple closed curves $\mathfrak{l}, \mathfrak{m}$ consisting of edges of $K$. Let $\lambda, \mu$ be oriented closed normal curves forming a basis for $H_1 (K)$. (Note $\lambda, \mu$ may be very different from $\mathfrak{l},\mathfrak{m}$!)

Let $v_0 \in V(K)$, and let $e_0$ an edge on $\partial K_0$. We then have several sets of circle packing equations, which we denote as follows:
\begin{itemize}
\item $\mathcal{E}_0$, the circle packing equations for $K_0$ (\refdef{circle_packing_eqns_disc});
\item $\mathcal{E}_{red}$, the reduced circle packing equations for $(K, \lambda, \mu, e_0, v_0)$ (\refdef{reduced_eqns_torus});
\item $\mathcal{E}$, the full set of circle packing equations for $(K, \lambda, \mu)$ (\refdef{circle_packing_eqns_torus}).
\end{itemize}
As the corners of $K$ are naturally in bijection with those of $K_0$, we can naturally identify the variables in all these sets of equations. From $\mathcal{E}_0$, we obtain $\mathcal{E}_{red}$ by removing the vertex equation at $v_0$ (if $v_0 \in V_{int}(K_0)$), and adding vertex equations along $\partial K_0 \setminus \{v_0\}$, edge equations along $\partial K_0 \setminus \{e_0\}$, and holonomy equations along $\lambda$ and $\mu$. From $\mathcal{E}_{red}$, $\mathcal{E}$ is obtained by adding the vertex equation at $v_0$ and the edge equation at $e_0$. Note $\mathcal{E}_0 \subset \mathcal{E}$ and $\mathcal{E}_{red} \subset \mathcal{E}$.

We introduce some notions and lemmas which will be useful in proving our results. As in \refsec{redundancy_rigidity}, let $V=|V(K)|$, $E=|E(K)|$ and $F=|F(K)|$.

\begin{lem}
\label{Lem:triangles_even}
$3F=2E$. In particular, $K$ contains an even number of triangles.
\end{lem}

\begin{proof}
Each edge lies on the boundary of two triangular faces.
\end{proof}

Given positive numbers $\{m_c\}_{c \in C(K)}$, we now define a function $\Theta \colon H_1 (K) \To \R/2\pi\Z$, which will keep track of the rotation of edges as we proceed around normal curves. Similar to \refsec{circle_packing_eqns_tori} and \refsec{torus_eqns_satisfied}, consider an oriented closed normal curve $\gamma$ on $K$, consisting of $n$ arcs $\gamma_j$ over $j \in \Z / n \Z$, concatenated in cyclic order, where $\gamma_j$ lies in triangle $\Delta_j$, cutting off corner $c_j$. Define $\delta_j = 1$ or $-1$ accordingly as $\gamma$ proceeds anticlockwise or clockwise around $c_j$. In each corner $c_j$ we have a parameter $m_{c_j}$. As usual we define angles $\theta_{c_j} \in (0,\pi)$ such that $m_\bullet = \cot(\theta_\bullet)/2$ and $\arg(m_\bullet + i) = \theta_\bullet/2$. We define
\begin{equation}
\label{Eqn:Theta_defn}
\Theta(\gamma) = 2 \arg \prod_{j=1}^{n} \left( m_{c_j} + i \right)^{\delta_j}
= 2 \sum_{j=1}^{n} \delta_j \arg \left( m_{c_j} + i \right) = \sum_{j=1}^{n} \delta_j \theta_{c_j}.
\end{equation}
(This quantity is in fact well defined modulo $4\pi$, but we only need it modulo $2\pi$.) 

\begin{lem}
\label{Lem:Theta_well_defined}
If the $\{m_c\}_{c \in C(K)}$ satisfy the vertex equations, then $\Theta(\gamma)$ only depends on the homology class of $\gamma$, yielding a well-defined map $\Theta \colon H_1 (K) \To \R/2\pi \Z$.
\end{lem}

\begin{proof}
If we homotope $\gamma$ across a vertex $v$, surrounded by corners $c_j$ with angles $\theta_j$ then we add $\delta \sum \theta_j$ to the value of $\Theta$, where $\delta = \pm 1$ accordingly as $\gamma$ proceeds anticlockwise or clockwise around $v$ after the homotopy. The vertex equation at $v$ is equivalent to $\sum \theta_j$ being an integer multiple of $2\pi$. Thus after the homotopy, $\Theta$ is unchanged modulo $2\pi$. Any oriented closed normal curve homologous to $\gamma$ is related by such homotopies.
\end{proof}

Since a concatenation of closed oriented normal curves yields angles which sum under $\Theta$, we observe that $\Theta$ is a homomorphism of additive groups. 

\begin{lem}
\label{Lem:holonomy_Theta_0}
Suppose the $\{m_c\}_{c \in C(K)}$ satisfy the vertex equations. Then they satisfy he holonomy equations if and only if $\Theta$ is identically zero.
\end{lem}

\begin{proof}
As discussed in \refsec{circle_packing_eqns_tori}, the holonomy equation for $\lambda$ is equivalent to $\prod_{j=1}^{n_\lambda} \left( m_{c_j^\lambda} + i \right)^{\delta_j^\lambda}$ being real, i.e. having argument an integer multiple of $\pi$. Since $\arg(m_\bullet + i) = \theta_\bullet/2$, this is equivalent to $\sum_{j=1}^n \delta_j^\lambda \theta_{c_j^\lambda}$ being an integer multiple of $2\pi$, i.e. $\Theta(\lambda) = 0$. A similar argument applies to $\mu$, so the holonomy equations are equivalent to $\Theta(\lambda) = \Theta(\mu) = 0$. As $\Theta$ is a homomorphism and $\lambda, \mu$ form a basis of $H_1 (K)$, this is equivalent to $\Theta \equiv 0$.
\end{proof}

\begin{proof}[Proof of \refprop{circle_packing_reduction_torus}]
Since $\mathcal{E}_{red} \subset \mathcal{E}$, it suffices to show that a positive solution $\{m_{c}\}_{c \in C(K)}$ of $\mathcal{E}_{red}$ also satisfies $\mathcal{E}$.

First we show the $m_\bullet$ satisfy the vertex equation at $v_0$. At each corner $c \in C(K)$, let $\theta_c$ be the unique angle in $(0, \pi)$ such that $m_c = \cot (\theta_c / 2)$. By \reflem{triangle_equation_angle_sum}, the triangle equation at $\Delta \in F(K)$ is equivalent to the angles in $\Delta$ summing to $\pi$. By \reflem{vertex_equation_angle_sum}, the vertex equation at $v \in V(K)$ is equivalent to the angles around $v$ summing to a multiple of $2\pi$.

As $\mathcal{E}_{red}$ includes all triangle equations, $\sum_{c \in C(K)} \theta_c = \pi F$, which by \reflem{triangles_even} is a multiple of $2\pi$. On the other hand, this sum is also the sum of the angles around all vertices. As $\mathcal{E}_{red}$ includes the vertex equation at every $v \in V(K) \setminus \{v_0\}$, the angles around every $v \in V(K) \setminus \{v_0\}$ sum to a multiple of $2\pi$. Thus, the angles around $v_0$ also sum to a multiple of $2\pi$, and the $m_\bullet$ satisfy the vertex equation at $v_0$. 

The $m_\bullet$ thus satisfy the circle packing equations $\mathcal{E}_0$ of $K_0$. Moreover, by \reflem{Theta_well_defined} then $\Theta$ is well defined, and as $\mathcal{E}_{red}$ includes the holonomy equations then \reflem{holonomy_Theta_0} implies $\Theta \equiv 0$.

As the $m_\bullet$ satisfy $\mathcal{E}_0$, \refsec{closed_disc} yields a circle packing of $K_0$ in the Euclidean plane, with realisation $\Phi_0 \colon K_0 \To \R^2$, angles $\{\theta_c\}_{c \in C(K_0)}$ as already introduced, and curvatures $\{\kappa_v\}_{v \in V(K_0)}$ are related to the $m_c$ by \refeqn{desired_relations}.

As in the proof of \reflem{torus_packings_give_solutions}, consider an oriented normal curve $\gamma$ in $K_0$ consisting of $n$ normal arcs $\gamma_j$, where each $\gamma_j$ proceeds from a point $p_{j-1}$ on an edge $e_{j-1}$ of a triangle $\Delta_j$, to a point $p_j$ on the edge $e_j$ of $\Delta_j$, cutting off corner $c_j$, with $\delta_j = \pm 1$ accordingly as $\gamma_j$ cuts off $c_j$ anticlockwise or clockwise. Orient each $e_j$ so that the direction of $\gamma$ and the direction of $e_j$ form an oriented basis. Then under the realisation $\Phi_0$, the direction of $e_j$ is obtained from the direction of $e_{j-1}$ by a rotation of $\delta_j \theta_{c_j}$. Thus the direction of the final edge $e_n$ is obtained from the direction of the initial $e_0$ by a rotation of $\sum_{j=1}^n \delta_j \theta_{c_j}$.

In particular, if $\gamma_0$ is an oriented normal curve in $K_0$ obtained from a closed oriented normal curve $\gamma$ in $K$ after cutting along $\mathfrak{l}$ and $\mathfrak{m}$, then under $\Phi_0$, $e_n$ is obtained from $e_0$ by a rotation of $\sum_{j=1}^n \delta_j \theta_{c_j} = \Theta(\gamma) = 0$. Thus $\Phi_0(e_0)$ and $\Phi_0(e_n)$ point in the same direction.

\begin{figure}
\begin{center}
\begin{tikzpicture}[scale=1.5]
 \draw (0,0) -- (1,1) node[midway,left] {$\mathfrak{m}_1^L$} node[right] {\textcolor{blue}{$p_1^L$}} -- (0.4,1.6) node[midway,left] {$\mathfrak{m}_2^L$} node[right, xshift=0.3em] {\textcolor{blue}{$p_2^L$}} -- (1.1,2.3) node[right] {\textcolor{blue}{$\vdots$}} -- (0.6,3) -- (1,3.5);
 \draw (3.5,0.1) -- (4.5,1.1) node[midway,right] {$\mathfrak{m}_1^R$} node[left,xshift=-0.3em] {\textcolor{blue}{$p_1^R$}} -- (3.9,1.7) node[midway,right] {$\mathfrak{m}_2^R$} node[left] {\textcolor{blue}{$p_2^R$}} --(4.6,2.4) node[left,xshift=-0.4em,yshift=0.3em] {$\textcolor{blue}{\vdots}$} --(4.1,3.1)--(4.5,3.6);
 \draw (1,3.5) -- (1.6,3.1) node[midway,above] {$\mathfrak{l}_1^T$} node[below] {\textcolor{blue}{$p_1^T$}} -- (2.1,3.5) node[midway,above] {$\mathfrak{l}_2^T$} node[below] {\textcolor{blue}{$p_2^T$}} -- (2.9,3.1) node[below] {$\textcolor{blue}{\ldots}$} --(4.5,3.6);
 \draw (0,0) -- (0.6,-0.4) node[midway,below] {$\mathfrak{l}_1^B$} node[above] {\textcolor{blue}{$p_1^B$}} -- (1.1,0) node[midway,below] {$\mathfrak{l}_2^B$} node[above] {\textcolor{blue}{$p_2^B$}} -- (1.9,-0.4) node[above] {$\textcolor{blue}{\cdots}$} -- (3.5,0.1);
 \filldraw[blue] (0,0) node[below left] {$p^{BL}$} circle (1pt);
 \filldraw[blue] (1,1) circle (1pt);
 \filldraw[blue] (0.4,1.6) circle (1pt);
 \filldraw[blue] (1.1,2.3) circle (1pt);
 \filldraw[blue] (0.6,3) circle (1pt);
 \filldraw[blue] (1,3.5) node [above left] {$p^{TL}$} circle (1pt);
 \filldraw[blue] (1.6,3.1) circle (1pt);
 \filldraw[blue] (2.1,3.5) circle (1pt);
 \filldraw[blue] (2.9,3.1) circle (1pt);
 \filldraw[blue] (3.5,0.1) node [below right] {$p^{BR}$} circle (1pt);
 \filldraw[blue] (4.5,1.1) circle (1pt);
 \filldraw[blue] (3.9,1.7) circle (1pt);
 \filldraw[blue] (4.6,2.4) circle (1pt);
 \filldraw[blue] (4.1,3.1) circle (1pt);
 \filldraw[blue] (4.5,3.6) node [above right] {$p^{TR}$} circle (1pt);
 \filldraw[blue] (0.6,-0.4) circle (1pt);
 \filldraw[blue] (1.1,0) circle (1pt);
 \filldraw[blue] (1.9,-0.4) circle (1pt);
 \draw[dashed,->] (-0.5,0) -- (0.5,3.5) node[midway, left] {$\mathfrak{m}^L$};
 \draw[dashed,->] (4.5,0.1) -- (5.5,3.6) node[midway, right] {$\mathfrak{m}^R$};
 \draw[dashed,->] (1,3.9) -- (4.5,4) node[midway,above] {$\mathfrak{l}^T$};
 \draw[dashed,->] (0,-0.7) -- (3.5,-0.6) node[midway,below ] {$\mathfrak{l}^B$};
 \node at (2.2,1.8) {$K_0$};
 \node at (0.5,0.5) {};
\end{tikzpicture}
\end{center}
 \caption{ Fundamental domain complex $K_0$, with boundary edge and vertices labelled.}
 \label{Fig:fundamental_domain_boundary_labels}
\end{figure}

We now label vertices and edges in $\partial K_0$ as illustrated in \reffig{fundamental_domain_boundary_labels}. The boundary of $K_0$ consists of two copies of $\mathfrak{l}$ and $\mathfrak{m}$, which we may draw with both copies of $\mathfrak{l}$ proceeding from left to right, and both copies of $\mathfrak{m}$ proceeding from bottom to top. Denote the bottom and top copies of $\mathfrak{l}$ as $\mathfrak{l}^B$ and $\mathfrak{l}^T$, and the left and right copies of $\mathfrak{m}$ as $\mathfrak{m}^L$ and $\mathfrak{m}^R$ respectively. Let $\mathfrak{l}, \mathfrak{m}$ consist of $n_{\mathfrak{l}}, n_{\mathfrak{m}}$ edges respectively. Denote the $j$th edge of $\mathfrak{l}$ and $\mathfrak{m}$ by $\mathfrak{l}_j$ and $\mathfrak{m}_j$ respectively, and the corresponding edges of $\mathfrak{l}^B, \mathfrak{l}^T, \mathfrak{m}^L, \mathfrak{m}^R$ by $\mathfrak{l}_j^B, \mathfrak{l}_j^T, \mathfrak{m}_j^L, \mathfrak{m}_j^R$ respectively, numbered in order. For $0 \leq j \leq n_\mathfrak{l}$, denote the vertices of $K_0$ along $\mathfrak{l}^B$ and $\mathfrak{l}^T$ in order by $p_j^B$ and $p_j^T$. Similarly, for $0 \leq j \leq n_\mathfrak{m}$, denote the vertices of $K_0$ along $\mathfrak{m}^L$ and $\mathfrak{m}^R$ in order by $p_j^L$ and $p_j^R$. Denote the bottom-left, bottom-right, top-left, top-right vertices by $p^{BL}, p^{BR}, p^{TL}, p^{TR}$, so $p_0^B = p_0^L = p^{BL}$, $p_{n_\mathfrak{l}}^B = p_0^R = p^{BR}$, $p_0^T = p_{n_\mathfrak{m}}^L = p^{TL}$, and $p_{n_\mathfrak{l}}^T = p_{n_{\mathfrak{m}}}^R = p^{TR}$. Thus each $\mathfrak{l}_j^B$ has endpoints $p_{j-1}^B, p_j^B$, each $\mathfrak{l}_j^T$ has endpoints $p_{j-1}^T, p_j^T$, each $\mathfrak{m}_j^L$ has endpoints $p_{j-1}^L, p_j^L$, and each $\mathfrak{m}_j^R$ has endpoints $p_{j-1}^R, p_j^R$.

For each integer $j$ from $1$ to $n_\mathfrak{l}$, there is a closed oriented normal curve $\gamma$ in $K$, homologous to $\mathfrak{m}$, which becomes an oriented normal curve $\gamma_0$ in $K_0$ from $\mathfrak{l}_j^B$ to $\mathfrak{l}_j^T$. As $\Theta(\gamma) = 0$ then the initial and final edges $\Phi_0 (\mathfrak{l}_j^B)$ and $\Phi_0 (\mathfrak{l}_j^T)$ point in the same direction. Similarly, for each $1 \leq j \leq n_\mathfrak{m}$, there is a closed oriented normal curve in $K$, homologous to $\mathfrak{l}$, yielding a normal curve in $K_0$ from $\mathfrak{m}_j^L$ to $\mathfrak{m}_j^R$, so that $\Phi_0(\mathfrak{m}_j^L)$ and $\Phi_0 (\mathfrak{l}_j^T)$ point in the same direction.

Thus the polygonal paths $\Phi_0(\mathfrak{m}^L)$ and $\Phi_0(\mathfrak{m}^R)$ consist of corresponding parallel segments, as do the polygonal paths $\Phi_0(\mathfrak{l}^B)$ and $\Phi_0(\mathfrak{l}^T)$.

Now consider the edge equations around $\partial K_0$. Using \reflem{edge_eqn_fact} we express these in terms of curvatures of circles. At each vertex $v$, $\Phi_0$ realises a circle with curvature $\kappa_v$. For convenience, we write $\kappa_j^B, \kappa_j^T, \kappa_j^L, \kappa_j^R$ for $\kappa_{p_j^{B}}, \kappa_{p_j^T}, \kappa_{p_j^L}, \kappa_{p_j^R}$ respectively, and $\kappa^{BL}, \kappa^{BR}, \kappa^{TL}, \kappa^{TR}$ to mean $\kappa_{p^{BL}}, \kappa_{p^{BR}}, \kappa_{p^{TL}}, \kappa_{p^{TR}}$ respectively.

Using \reflem{edge_eqn_fact}, the edge equations of $\mathfrak{l}_j$ and $\mathfrak{m}_j$ respectively assert that
\[
\frac{\kappa_{j-1}^B}{\kappa_{j-1}^T} = \frac{\kappa_j^B}{\kappa_j^T}
\quad \text{and} \quad
\frac{\kappa_{j-1}^L}{\kappa_{j-1}^R} = \frac{\kappa_j^L}{\kappa_j^R}.
\] 
The edge equations of $\mathfrak{l}_1, \ldots, \mathfrak{l}_{n_\mathfrak{l}}$ respectively assert the chain of equalities
\begin{equation}
\label{Eqn:chain_of_l_edge_eqns}
\frac{\kappa_0^B}{\kappa_0^T} = \frac{\kappa_1^B}{\kappa_1^T} = \cdots = 
\frac{\kappa_{j-1}^B}{\kappa_{j-1}^T} = \frac{\kappa_j^B}{\kappa_j^T}
= \cdots = \frac{\kappa_{n_\mathfrak{l}}^B}{\kappa_{n_\mathfrak{l}}^T},
\end{equation}
and the edge equations of $\mathfrak{m}_1, \ldots, \mathfrak{m}_{n_\mathfrak{m}}$ assert the chain of equalities
\begin{equation}
\label{Eqn:chain_of_m_edge_eqns}
\frac{\kappa_0^L}{\kappa_0^R} = \frac{\kappa_1^L}{\kappa_1^R} = \cdots = 
\frac{\kappa_{j-1}^L}{\kappa_{j-1}^R} = \frac{\kappa_j^L}{\kappa_j^R}
= \cdots = \frac{\kappa_{n_\mathfrak{m}}^L}{\kappa_{n_\mathfrak{m}}^R}.
\end{equation}
Note that the equality of the first and last entries in \refeqn{chain_of_l_edge_eqns} is $\kappa^{BL}/\kappa^{TL} = \kappa^{BR}/\kappa^{TR}$, and the equality of the first and last expressions in \refeqn{chain_of_m_edge_eqns} is the equivalent $\kappa^{BL}/\kappa^{BR} = \kappa^{TL}/\kappa^{TR}$.

Suppose the missing edge equation in $\mathcal{E}_{red}$ is for $\mathfrak{m}_j$, i.e. $e_0 = \mathfrak{m}_j$. Then all equalities of \refeqn{chain_of_l_edge_eqns} hold, and all equalities except the $j$th of \refeqn{chain_of_m_edge_eqns} hold. However, since the first and last expressions in \refeqn{chain_of_l_edge_eqns} are equal, the first and last expressions in \refeqn{chain_of_m_edge_eqns} are also equal, and hence in fact all the expressions in \refeqn{chain_of_m_edge_eqns} are equal. Hence the $j$th equality of \refeqn{chain_of_m_edge_eqns} also holds. 

Similarly, if $e_0 = \mathfrak{l}_j$, then all equalities of \refeqn{chain_of_m_edge_eqns} hold, and all equalities except the $j$th of \refeqn{chain_of_l_edge_eqns} hold. Since the first and last expressions of \refeqn{chain_of_m_edge_eqns} are equal, the first and last expressions of \refeqn{chain_of_l_edge_eqns} are also equal, so in fact all expressions in \refeqn{chain_of_l_edge_eqns} are equal, and the $j$the equality of \refeqn{chain_of_l_edge_eqns} holds.

In other words, in any case, the edge equation for $e_0$ holds, and the $m_\bullet$ satisfy $\mathcal{E}$.
\end{proof}

We can now prove the main theorem. In \refsec{torus_eqns_satisfied}, given a Euclidean circle packing of $K$, we obtained positive parameters $\{m_c\}_{c \in C(K)}$, depending only on the conformal class of the circle packing, which satisfy the circle packing equations $\mathcal{E}$ for $(K, \lambda, \mu)$, and are related to the curvatures $\{\kappa_v\}_{v \in V(K)}$ and angles $\{\theta_c\}_{c \in C(K)}$ by \refeqn{desired_relations}. It remains to show that this process can be reversed, constructing circle packings from positive real solutions of $\mathcal{E}$, to provide a bijection between these solutions and conformal classes of Euclidean circle packings of $K$.

\begin{proof}[ Proof of \refthm{circle_packing_eqs_general} and \refprop{m_meaning} when $K \cong T^2$]
Since $\mathcal{E}_{red} \subset \mathcal{E}$, we apply the arguments of the previous proof: we form a fundamental domain complex $K_0$ by cutting $K$ along oriented $\mathfrak{l}, \mathfrak{m}$; as the $\{m_c\}_{c \in C(K)}$ satisfy $\mathcal{E}_0 \subset \mathcal{E}$, we obtain a circle packing of $K_0$ with realisation $\Phi_0 \colon K_0 \To \R^2$ and angles $\{\theta_c\}_{c \in C(K_0)}$ and curvatures $\{\kappa_v\}_{v \in V(K_0)}$ satisfying \refeqn{desired_relations}; then $\partial K_0$ is as in \reffig{fundamental_domain_boundary_labels}, with the polygonal paths $\Phi_0(\mathfrak{m}^L), \Phi_0(\mathfrak{m}^R)$ consisting of corresponding parallel segments, and similarly for $\Phi_0(\mathfrak{l}^B)$ and $\Phi_0(\mathfrak{l}^T)$; using edge equations around $\partial K_0$, the equalities \refeqn{chain_of_l_edge_eqns} and \refeqn{chain_of_m_edge_eqns} hold.

By \refeqn{chain_of_l_edge_eqns} there is a positive real constant $A_\mathfrak{m}$ such that for each $1 \leq j \leq n_{\mathfrak{l}}$ we have $\kappa_j^B = A_\mathfrak{m} \kappa_j^T$. Similarly by \refeqn{chain_of_m_edge_eqns} there is a positive constant $A_\mathfrak{l}$ such that for each $1 \leq j \leq n_{\mathfrak{m}}$, we have $\kappa_j^L = A_\mathfrak{l} \kappa_j^R$. Writing $r_\bullet^\bullet = 1/\kappa_\bullet^\bullet$ for the corresponding radius of each circle, we then have
\[
r_j^T = A_\mathfrak{m} r_j^B \quad \text{for all $1 \leq j \leq n_\mathfrak{l}$}, \quad \text{and} \quad
r_j^R = A_\mathfrak{l} r_j^L \quad \text{for all $1 \leq j \leq n_\mathfrak{m}$}.
\]
The length of each edge of $\partial K_0$ under $\Phi_0$ is the sum of two such radii. Thus each oriented edge $\Phi_0 (\mathfrak{l}_j^T)$ is parallel to and $A_\mathfrak{m}$ times the length of $\Phi_0 (\mathfrak{l}_j^B)$, and each $\Phi_0 (\mathfrak{m}_j^R)$ is parallel to and $A_\mathfrak{l}$ times the length of $\Phi_0 (\mathfrak{m}_j^L)$. In other words, $\Phi_0 (\mathfrak{l}^T)$ is obtained from $\Phi_0 (\mathfrak{l}^B)$ by a dilation of $A_\mathfrak{m}$ and a translation; and $\Phi_0(\mathfrak{m}^R)$ is obtained from $\Phi_0(\mathfrak{m}^L)$ by a dilation of $A_\mathfrak{l}$ and a translation.

Thus the quadrilateral formed by the four corners $p^{BL}, p^{BR}, p^{TL}, p^{TR}$ has its opposite sides parallel, its top side $A_\mathfrak{m}$ times the length of its bottom side, and its right side $A_\mathfrak{l}$ the length of its left side. But as opposite sides are parallel, we have a parallelogram, hence $A_\mathfrak{l} = A_\mathfrak{m} = 1$. 

Hence the polygonal sides of $\Phi_0 (K_0)$ are related by translations, and $\Phi_0$ extends to a realisation $\widetilde{\Phi} \colon \widetilde{K} \To \R^2$ of a Euclidean circle packing of the universal cover $\widetilde{K}$ of $K$ in $\R^2$, such that deck transformations of $K$ correspond to translations under $\widetilde{\Phi}$. Hence $\widetilde{\Phi}$ is the development of the realisation $\Phi \colon K \To \R^2$ of a circle packing of $K$ on a Euclidean torus $T^2$, having the same angles $\theta_\bullet$ and curvatures $\kappa_\bullet$, related to the $m_\bullet$ by \refeqn{desired_relations}. As the circle packing of $K_0$ is unique up to conformal equivalence, the circle packing of $K$ so obtained is also unique up to conformal equivalence.
\end{proof}

\begin{proof}[Proof of \refthm{circle_packing_eqs_general_unbranched} when $K \cong T^2$]
By \refsec{closed_disc}, if the $m_\bullet$ satisfy the unbranched circle packing equations, then the circle packing of $K_0$ obtained in the above argument is unbranched, and its realisation $\Phi_0$ is a homeomorphism of $K_0$ onto a ``polygonal parallelogram". The realisation $\widetilde{\Phi} \colon \widetilde{K} \To \R^2$ is then a homeomorphism, so the circle packing of $K$ obtained has no branching.
\end{proof}

{\flushleft \textbf{Conflict of Interests statement.} }
On behalf of all authors, the corresponding author states that there is no conflict of interest. 

{\flushleft \textbf{Data availability statement.} }
This work has no associated data.

\bibliographystyle{amsplain}
\bibliography{refs.bib}

\end{document}